\documentclass[12pt,english]{article}
\usepackage[T1]{fontenc}
\usepackage[latin9]{inputenc}
\usepackage[a4paper]{geometry}
\usepackage{graphicx}
\usepackage{color, soul}
\usepackage{indentfirst}
\usepackage{enumerate}
\usepackage{amsmath}
\usepackage{amsthm}
\usepackage{amssymb}
\usepackage{mathtools}
\usepackage{tikz}
\usetikzlibrary{shapes.geometric, arrows,shapes.misc}
\usepackage{leftindex}
\usepackage{comment}
\usepackage{subcaption}
\usepackage{epsfig}
\usepackage[arrow]{xy}
\usepackage[linktocpage=true]{hyperref}
\usepackage{tocloft} 
\usepackage{babel}
\usepackage[T1]{fontenc}
\graphicspath{ {./images/} }

\tikzstyle{startstop} = [rectangle, rounded corners, minimum width=3cm, minimum height=1cm,text centered, draw=black, fill=red!30]
\tikzstyle{process} = [rectangle, minimum width=3cm, minimum height=1cm, text centered, draw=black, fill=orange!30]
\tikzstyle{decision} = [chamfered rectangle, minimum width=3cm, minimum height=1cm, text centered, draw=black, fill=green!30]
\tikzstyle{arrow} = [thick,->,>=stealth]

\newtheorem{thm}{Theorem}[section]
\newtheorem{definition}[thm]{Definition}
\newtheorem{remark}[thm]{Remark}

\usepackage{authblk}
\makeatletter
\def\thanks#1{\protected@xdef\@thanks{\@thanks
        \protect\footnotetext{#1}}}
\makeatother

\title{Jacobi Hamiltonian Integrators: construction and applications}
\author{Ad\'erito Ara\'ujo}
\author{Gon\c calo Inoc\^encio Oliveira\thanks{g.inoc.oliveira@gmail.com}}
\author{Jo\~ao Nuno Mestre}
\affil{CMUC, University of Coimbra, Department of Mathematics, Portugal.}

\date{}
\begin{document}
	
\maketitle
\begin{abstract}
We propose a systematic framework for constructing geometric integrators for Hamiltonian systems on Jacobi manifolds. By combining Poissonization of Jacobi structures with homogeneous symplectic bi-realizations, Jacobi dynamics are lifted to homogeneous Poisson Hamiltonian systems, enabling the construction of structure-preserving Jacobi Hamiltonian integrators. The resulting schemes are constructed explicitly and applied to a range of examples, including contact Hamiltonian systems and classical models. Numerical experiments highlight their qualitative advantages over standard integrators, including better preservation of geometric structure and improved long-time behavior.
\end{abstract}

\section{Introduction}

Classical Hamiltonian systems are systems of ordinary differential equations that model a wide range of dynamical behaviors and arise from a geometric construction based on a smooth function, called the Hamiltonian, defined on a manifold equipped with a suitable geometric structure \cite{dwivedi2019hamiltonian,libermann1987symplectic}. In classical Hamiltonian mechanics, this manifold is the phase space endowed with a symplectic structure. Symplectic geometry allows one to associate to the Hamiltonian function an intrinsically defined Hamiltonian vector field, whose integral curves describe the time evolution from given initial conditions. The collection of these trajectories defines the Hamiltonian flow, which preserves the Hamiltonian, corresponding to conservation of energy.

Except in special cases, Hamiltonian flows cannot be computed analytically. However, their geometric structure can be exploited to construct accurate numerical approximations using geometric integrators, numerical methods designed to preserve qualitative geometric properties of the exact flow \cite{Hairer2006N}. A prominent example is provided by symplectic integrators, which preserve the symplectic structure and consequently exhibit excellent long-time behavior \cite{sanz1994numerical}. Although such methods do not necessarily preserve the original Hamiltonian exactly, they typically conserve a nearby modified Hamiltonian, leading to bounded energy errors over long time intervals.

The scope of Hamiltonian methods has expanded considerably through the introduction of more general geometric frameworks. Poisson geometry \cite{crainic2021lectures,Weinstein1983} generalizes symplectic geometry by allowing degeneracies and naturally arises in conservative systems with symmetries or singularities. Contact geometry provides a natural setting for modeling dissipative and time-dependent systems, with applications ranging from mechanics to thermodynamics \cite{Bravetti_19,Bravetti_2017,deLeonSardon2017,MRUGALA1991109}. Jacobi geometry \cite{BruceGrabowskaGrabowski2017,Kirillov,LICHNEROWICZ1977}, which encompasses both Poisson and contact structures, offers a unified framework for the description of conservative and dissipative dynamics, as well as systems with symmetries and singularities \cite{deLeonetal2023,LiuZhang2011,SardondeLucasHerranz2015,ZapataPhD,Zapata2020}. In these settings, the geometry again determines how a Hamiltonian function induces a system of differential equations through an associated Hamiltonian vector field. In contrast to the symplectic case, Jacobi and contact geometries encode a controlled evolution of the Hamiltonian along trajectories, reflecting dissipation or external forcing.

In the Poisson case, however, the preservation of the Poisson structure by geometric integrators alone does not guarantee conservation of the Hamiltonian. To address this limitation, Poisson Hamiltonian integrators (PHIs) were introduced in \cite{COSSERAT2023}, generating discrete trajectories that coincide exactly with the flow of a suitably chosen time-dependent Hamiltonian approximating the original one. Building on these ideas, we proposed in \cite{araujo2025jacobihamiltonianintegrators} a class of Jacobi Hamiltonian integrators (JHIs) for Jacobi Hamiltonian systems. These methods preserve the underlying Jacobi structure and produce discrete trajectories that are exact solutions of a modified time-dependent Jacobi Hamiltonian system. As is typical of structure-preserving schemes, JHIs exhibit favorable long-time behavior and avoid the accumulation of large global errors.

The key contributions of this work are as follows. First, we develop a systematic framework for constructing JHIs by exploiting the geometric relationship between Jacobi manifolds and homogeneous Poisson manifolds, using Poissonization and homogeneous symplectic bi-realizations \cite{Monica_Blaga_2020,CMS2022,CrainicMarcut2011,CrainicZhu2004}. Second, we introduce an explicit recursive algorithm for generating JHIs of arbitrary order, based on the Magnus expansion and exact homogeneous Lagrangian bisections. The resulting flows produce well-defined Jacobi maps, preserve intrinsic homogeneity, and, as confirmed by backward error analysis, coincide up to the predicted order with the exact flow of a modified Jacobi Hamiltonian system, ensuring long-time qualitative fidelity. Finally, we demonstrate the effectiveness of our approach through numerical experiments on both low-dimensional illustrative examples and classical models, highlighting the accuracy and strong geometric consistency of JHIs compared with standard numerical schemes.

The paper is organized as follows. Section~\ref{sec:construction} presents the explicit construction of Jacobi Hamiltonian integrators introduced in \cite{araujo2025jacobihamiltonianintegrators}. The construction is based on the Poissonization of Jacobi manifolds and on a Magnus expansion applied to the Jacobi Hamiltonian flow, formulated in terms of homogeneous Lagrangian bisections. In addition, Section~\ref{sec:construction} provides a qualitative analysis of the resulting schemes, establishing their structure-preserving properties, including arbitrary-order accuracy, backward error behavior through a modified Hamiltonian, and exact preservation of homogeneity and Jacobi invariants. Section~\ref{sec:numerical} summarizes the construction in algorithmic form and applies the proposed integrators to several illustrative examples, including contact Hamiltonian systems, low-dimensional test problems, and classical models. Numerical experiments and comparisons with standard integrators are discussed throughout, highlighting the qualitative and quantitative advantages of the proposed approach.


\section{Construction of Jacobi Hamiltonian integrators}\label{sec:construction}

Jacobi geometry generalizes both Poisson and contact geometry, providing a natural framework for modeling time-dependent, dissipative, and thermodynamic systems \cite{Bravetti_19, Bravetti_2017,deLeonetal2023,LiuZhang2011,SardondeLucasHerranz2015,ZapataPhD,Zapata2020}. Standard symplectic or Poisson integrators do not directly apply in this more general setting, which motivates the construction of specialized geometric integrators for Jacobi systems.

\subsection{Construction of the method}
The construction of Jacobi Hamiltonian Integrators (JHI) begins by introducing a homogeneous symplectic bi-realization of the Poissonized manifold, providing the geometric setting in which Hamiltonian flows can be represented as Lagrangian bisections.

A Hamiltonian system on a Jacobi manifold is lifted, via Poissonization, to a homogeneous Poisson Hamiltonian system on the associated principal $\mathbb{R}^\times$-bundle. On this manifold, a symplectic bi-realization can be constructed using a homogeneous Poisson spray, yielding a local homogeneous symplectic groupoid \cite{Monica_Blaga_2020,CrainicMarcut2011,CrainicZhu2004}. Within this framework, Hamiltonian dynamics are conceptually encoded by smooth families of homogeneous Lagrangian bisections governed by a homogeneous Hamilton--Jacobi equation. The maps $\alpha$ and $\beta$ of the bi-realization, once defined, will connect the cotangent bundle with the manifold and ensure the preservation of both the symplectic structure and the homogeneity.

Building on this geometric framework, the numerical integrator is constructed iteratively by approximating the generating functions of these Lagrangian bisections. For time-dependent Hamiltonians, the Magnus expansion provides a systematic tool for generating the coefficients of this approximation, expressing the flow as a single exponential of a Hamiltonian vector field. This representation naturally respects the Hamiltonian Lie algebra structure and will be essential for the backward error analysis in the following section, which establishes the preservation of key geometric features of Jacobi dynamics.

\subsubsection{Poissonization of Jacobi Hamiltonian systems}
A Jacobi structure on a smooth manifold $J$ is given by a bivector field $\Lambda$ and a vector field $E$ on $J$ such that $[\Lambda, \Lambda] = 2 E \wedge \Lambda $ and $[E, \Lambda] = 0$. We call $(J,\Lambda,E)$ a Jacobi manifold. Given a Hamitonian, i.e., a smooth function $H$ on $J$, the associated Jacobi Hamiltonian vector field, determining a Hamiltonian system, is 
\begin{equation}\label{JH_vectorfield}
X_H = \Lambda(\cdot,dH) - H E(\cdot).
\end{equation}

Let $(J, \Lambda, E)$ be a Jacobi manifold with coordinates $x$. The Poissonization procedure \cite{LICHNEROWICZ1977} associates to $J$ a homogeneous Poisson manifold $(P, \Pi, h_z)$, where $\Pi$ is a $-1$-homogeneous Poisson structure 
\begin{equation}\label{poisson_structure}
\Pi = \frac{1}{t}\Lambda + \frac{\partial}{\partial t} \wedge E,\\
\end{equation}
and $h_z$ defines the homogeneity action for $z \in \mathbb{R}^\times$
\[
h_z(x,t) = (x, z t), \qquad T^*h_z(x,t,\xi_x,\xi_t) = (x, zt, z\xi_x, \xi_t).
\]
Under this correspondence, a Hamiltonian system on $J$ lifts to a homogeneous Hamiltonian system on $P$ with a 1-homogeneous Hamiltonian $\hat H$, i.e., $h_z^* \hat H = z \hat H$.

\subsubsection{Homogeneous symplectic bi-realization}

A homogeneous symplectic bi-realization provides the geometric framework in which Hamiltonian flows are represented as Lagrangian bisections. The bi-realization defines the maps
\[
\alpha, \beta: \mathcal{U}\subset T^*P \to P
\]
that connect the cotangent bundle with the manifold while preserving both the Poisson structure and the homogeneity:
\[
h_z \circ \alpha = \alpha \circ T^*h_z, \qquad h_z \circ \beta = \beta \circ T^*h_z.
\]

\subsubsection{Iterative construction of the method} \label{sec: explicit construction}

Once the bi-realization is defined, the iterative algorithm produces a discrete flow that approximates the Hamiltonian dynamics on the Poissonized manifold. Let $\hat{H}$ be the homogeneous Hamiltonian, and let $S_i$ denote the $i$-th term in the iterative construction. The algorithm of order $k$ proceeds recursively:

\begin{enumerate}
    \item[1.] Set $S_1(x) = \hat{H}(x)$;
    \item[2.] Compute
    \[
        S_{i+1}(x) = \frac{1}{(i+1)!} \left. \frac{d^i}{ds^i} \right|_{s=0} \hat{H}\left( \alpha \left(d_x S_s^{(i)} \right) \right), 
        \qquad S_s^{(i)} = \sum_{j=1}^i s^j S_j;
    \]
    \item[3.] For a time step $\Delta s$, find $y_n$ such that
    \[
        \alpha\left(y_n, \sum_{i=1}^k \Delta s^i \nabla_x S_i(y_n) \right) = x_n;
    \]
    \item[4.] Update
    \[
        x_{n+1} := \beta\left(y_n, \sum_{i=1}^k \Delta s^i \nabla_x S_i(y_n) \right).
    \]
\end{enumerate}

Here, each $S_i$ is the $i$-th coefficient of the Taylor expansion of $S_s$, which induces a Hamiltonian flow on the original manifold using $\alpha^{-1}$ and $\beta$.  

The explicit construction of the JHI follows that of the Poisson Hamiltonian Integrator \cite{cosserat2023numerical}; the key point is to verify that step 2 preserves the homogeneous structure.

\begin{thm}\label{thm: S_i homogeneous}
The recursion defined in steps 1 and 2 preserves the homogeneous structure; that is, all generated $S_i$ are 1-homogeneous.
\end{thm}

\begin{proof}
    We proceed by induction on $k$. For $k=1$, we have $S_1(x,t) = H(x,t)$, which is clearly 1-homogeneous. Assume now that $S_1, \dots, S_k$ are all 1-homogeneous. Consider $S_s^{(k)} = \sum_{j=1}^k s^j S_j$; by the induction hypothesis, $S_s^{(k)}$ is also 1-homogeneous.  

We claim that $d(S_s^{(k)} \circ h_z) = T^*h_z \circ d S_s^{(k)}$. Indeed, for the $x$-components, one has
\[
    \frac{\partial S_s^{(k)}}{\partial x}(x, zt) = z \frac{\partial S_s^{(k)}}{\partial x}(x, t),
\]
and for the $t$-component, using the chain rule,
\[
    \frac{d}{dt}S_s^{(k)} (x,zt) = z \frac{\partial S^{(k)}_s}{\partial t}(x,zt). 
\]

On the other hand, $\frac{d}{dt}S_s^{(k)} (x,zt) = \frac{d}{dt}\left(zS_s^{(k)} (x,t) \right) = z \frac{\partial S^{(k)}_s}{\partial t}(x,t)$. So,
\[\frac{\partial S_s^{(k)}}{\partial t} (x,zt) =\frac{\partial S_s^{(k)}}{\partial t} (x,t).\]

Finally, $S_{k+1}$ is 1-homogeneous because both $\alpha$ and $dS_s^{(k)}$ are homogeneous:
\begin{align*}
    h_z^* S_{k+1}(x,t) &= \frac{1}{(k+1)!} \frac{d^k}{ds^k}\Big|_{s=0} \hat{H}\Big( \alpha(d_x S_s^{(k)} \circ h_z) \Big)\\
    &= \frac{1}{(k+1)!} \frac{d^k}{ds^k}\Big|_{s=0} \hat{H}\Big( \alpha(T^*h_z \circ d_x S_s^{(k)}) \Big)\\
    &= \frac{1}{(k+1)!} \frac{d^k}{ds^k}\Big|_{s=0} \hat{H}\Big( h_z \circ \alpha(d_x S_s^{(k)}) \Big)\\
    &= z \frac{1}{(k+1)!} \frac{d^k}{ds^k}\Big|_{s=0} \hat{H}\Big( \alpha(d_x S_s^{(k)}) \Big)\\
    &= z S_{k+1}(x,t), 
\end{align*}
which concludes the proof.
\end{proof} 

This shows that the same algorithm constructed in \cite{cosserat2023numerical} can be directly applied in the Jacobi Hamiltonian Integrator framework to build the numerical method.

\subsection{Structure-preserving properties}
\label{sec:qualitative}

Having constructed the Jacobi Hamiltonian Integrator, we now analyze its qualitative and structure-preserving properties, including high-order accuracy, backward error behavior, and homogeneity. A central role is played by the Magnus expansion, which provides a systematic framework for representing the flow as a single exponential of a Hamiltonian vector field and for constructing higher-order integrators.

\subsubsection{Magnus expansion}

For time-dependent Hamiltonians, it is often advantageous to represent the flow map as an exponential of a single Hamiltonian vector field. The Magnus expansion \cite{BLANES2009151,Magnus1954,JAOteo_1991} provides such a representation, capturing the geometric structure of the dynamics through nested Poisson brackets.

Let $(M, \omega_\mathrm{can})$ be a canonical symplectic manifold with coordinates $x=(q^i,p_i)$, and let the Poisson bracket of two smooth functions $f,g \in C^\infty(M)$ be defined by
\[
\{f,g\} = \frac{\partial f}{\partial q^i}\frac{\partial g}{\partial p_i} - \frac{\partial f}{\partial p_i}\frac{\partial g}{\partial q^i}.
\]
The Hamiltonian vector field associated with $f$ acts on functions as $X_f g = \{f,g\}$. Nested Poisson brackets are denoted by
\[
\mathrm{ad}_f^0 g = g, \quad
\mathrm{ad}_f^n g = \{f, \mathrm{ad}_f^{n-1} g\}, \quad n\ge 1,
\]
with the corresponding Lie series
\[
\mathrm{Ad}_f g = \exp(\mathrm{ad}_f) g = \sum_{n=0}^\infty \frac{1}{n!} \mathrm{ad}_f^n g.
\]

For a time-dependent Hamiltonian $H(x,t)$ with flow $\xi(t)$ satisfying $\dot{\xi}(t) = X_H(\xi(t),t)$, the Magnus expansion seeks a representation of the flow operator $M(t)$ as
\[
M(t) = \mathrm{Ad}_{\Omega(t)} = \exp(\mathrm{ad}_{\Omega(t)}),
\]
where $\Omega(t)$ is a time-dependent Hamiltonian function determined recursively from the nested Poisson brackets of $H$. Truncating the series at any order provides an approximation of the flow as a single exponential of a Hamiltonian vector field, ensuring that each truncated Magnus integrator defines a symplectic map. This makes the Magnus expansion a natural tool for constructing high-order geometric integrators.

\begin{remark}
For a given step size $\epsilon$, the operator $\mathcal{M}_\epsilon(H) = \Omega(\epsilon)$ is called the \emph{Magnus series} of $H$ and encodes all higher-order corrections due to the non-commutativity of the time-dependent Hamiltonian.
\end{remark}

\begin{thm}[Magnus \cite{Magnus1954}]
The solution of
\[
\dot M = M X_{H_0}, \qquad M(0) = I,
\]
can be expressed as
\[
M(t) = \mathrm{Ad}_{\Omega(t)},
\]
where
\[
\dot\Omega = \frac{\mathrm{ad}_{\Omega}}{\mathrm{Ad}_{\Omega}-1} H_0 = \sum_{n=0}^{\infty} \frac{B_n}{n!} \mathrm{ad}_{\Omega}^n H_0, 
\qquad \Omega(0) = 0,
\]
and $B_n$ are the Bernoulli numbers.
\end{thm}

The Magnus expansion provides a hierarchy of approximations: each successive term accounts for the non-commutativity of the time-dependent Hamiltonian. Truncation at order $k$ directly controls the local accuracy of the integrator and ensures the preservation of the underlying symplectic or Jacobi structure.

\subsubsection{Backward error analysis}

Backward error analysis provides a rigorous framework to study the long-time behavior of geometric integrators. Rather than estimating the global error directly, one constructs a \emph{modified Hamiltonian system} whose exact flow coincides with the numerical method. This approach clarifies how geometric structures are preserved by the integrator \cite{Hairer2006N}.

For Jacobi Hamiltonian Integrators, there is a one-to-one correspondence with homogeneous Poisson Hamiltonian Integrators (PHI) \cite{araujo2025jacobihamiltonianintegrators}, which allows the backward error analysis to be performed in the homogeneous Poisson setting. This simplifies the construction of the modified Hamiltonian while preserving the geometric structure.

\begin{definition}
A \emph{numerical integrator} on a manifold $P$ is a smooth family of diffeomorphisms 
\[
\varphi = (\varphi_\epsilon)_{\epsilon \in I}, 
\]
where $I \subset \mathbb{R}$ is an interval containing $0$, such that the map $\epsilon \mapsto \varphi_\epsilon(x)$ is smooth for all $x \in P$.
\end{definition}

\begin{definition}
The integrator $\varphi$ is said to have \emph{order $k$} with respect to a vector field $X \in \mathfrak{X}(P)$ if
\[
\varphi_\epsilon(x) = \phi_X^\epsilon(x) + o(\epsilon^k),
\]
where $\phi_X^\epsilon$ denotes the exact flow of $X$.
\end{definition}

\begin{definition}
    Given a Poisson (resp. Jacobi) manifold and a Hamiltonian vector field $X_H$, an integrator $\varphi_\epsilon$ of order $k$ for $X_H$ is said to be a Poisson (resp. Jacobi) integrator if $\varphi_\epsilon$ is a Poisson (Jacobi) diffeomorphism.
\end{definition}

\begin{definition}
    Let $X_H$ be a Poisson (resp. Jacobi) Hamiltonian vector field. An integrator of order $k$, $\varphi_\epsilon$, is said to be a Poisson (resp. Jacobi) Hamiltonian integrator for $H$ if:
    \begin{enumerate}
        \item $\varphi_\epsilon$ is a Poisson (resp. Jacobi) diffeomorphism;
        \item there exists a time-dependent Hamiltonian $(H_s)_s$ such that:
        \begin{enumerate}
            \item $H_s-H = o(\epsilon^k)$;
            \item $\varphi_\epsilon = \phi^\epsilon_H + o(\epsilon^k).$
        \end{enumerate}
    \end{enumerate}
\end{definition}

\begin{remark}
    We can also define a homogeneous Poisson Hamiltonian integrator if it is a Poisson Hamiltonian integrator and also preserves the homogeneous structure.
\end{remark}

Let $(P, \Pi, h_z)$ be a homogeneous Poisson manifold and $\hat{H}$ the Poissonized Hamiltonian. Its Hamiltonian vector field $X_{\hat{H}}$ generates the flow $\phi_{\hat{H}}^s$. Given a homogeneous symplectic bi-realization $(\alpha, \beta)$ and the unit map $\sigma$ of the local symplectic groupoid integrating $P$ (having $\alpha$ and $\beta$ as source and target), an exact family of homogeneous Lagrangian bisections $L_s$ induces the Poisson diffeomorphism
\[
\varphi_{L_s} = \beta \circ \alpha^{-1}_{|L_s}.
\]
Choosing the exact Lagrangian bisections 
\[
L_s = \phi^s_{\alpha^* \hat{H}}(\sigma(P))
\] 
ensures that 
\[
\varphi_{L_s} = \phi^s_{\hat{H}}.
\]
Approximating these bisections therefore induces an approximate Hamiltonian flow.

The key result of \cite{COSSERAT2023} shows that if the Magnus series of the variation functions coincides with $\epsilon\hat H$ up to order $k$, then the induced diffeomorphisms define a Poisson Hamiltonian integrator of order $k$. Moreover, the associated modified Hamiltonian can be computed explicitly.

\begin{thm}[\cite{COSSERAT2023}]\label{thm: modified Ham}
Consider a homogeneous symplectic bi-realization $(\mathcal{U},\alpha,\beta)$ of a homogeneous Poisson manifold $(P,\Pi,h_z)$ and a homogeneous Hamiltonian $\hat H$. Let $(S_i)_{i\in\mathbb{N}}$ be the generating functions defined in Section~\ref{sec: explicit construction}. Then the Poisson automorphisms associated with the Lagrangian bisections $d(S_s^{(k)})$ define a homogeneous Poisson Hamiltonian integrator of order $k$ for $\hat H$, with homogeneous modified Hamiltonian $\hat{H}_s$ satisfying
\[
\mathcal{M}_\epsilon(\hat{H}_s)=\epsilon\hat H + o(\epsilon^k).
\]
\end{thm}

This result demonstrates that Jacobi Hamiltonian Integrators retain the fundamental nature of classical geometric integrators: rather than approximating the solution of the original system, they provide the exact flow of a \emph{modified Hamiltonian system} (cf. Theorem~\ref{thm: modified Ham}). Consequently, the numerical method inherits the key structural properties of Hamiltonian dynamics.

In particular, the integrator preserves the Jacobi structure, the level sets of the Poissonized Hamiltonian (and therefore the correct dissipation of the Jacobi Hamiltonian, see the next subsection), and the Jacobi Casimir functions up to the order of the method:

\begin{thm}[Structure Preservation]
Jacobi Hamiltonian Integrators of order $k$ preserve the homogeneous Poisson structure, the Poissonized Hamiltonian level sets, and the Jacobi Casimir functions up to order $k$.
\end{thm}

\begin{proof}
Since the modified Hamiltonian coincides with $\hat H$ up to order $k$, its Hamiltonian flow preserves all associated geometric invariants to the same order.
\end{proof}

\subsubsection{Dissipation of energy in Jacobi Hamiltonian systems}

The Poissonization of a Jacobi structure $(\Lambda,E)$ gives the homogeneous Poisson structure \eqref{poisson_structure} defined on the extended manifold with coordinate $t\in\mathbb{R}^\times$. The Jacobi Hamiltonian vector field for a Hamiltonian $H$ is \eqref{JH_vectorfield}, and a homogeneous Hamiltonian for $\Pi$ inducing the Jacobi dynamics is $\hat H = tH$.

The Hamiltonian vector field of $\hat H$ is
\begin{align*}
X_{\hat H} 
&= \Pi(\cdot,d\hat H) \\
&= \frac{1}{t}\Lambda(\cdot,d\hat H) + \frac{\partial}{\partial t}\wedge E(\cdot,d\hat H) \\
&= \Lambda(\cdot,dH) - H E(\cdot) + tE(H)\frac{\partial}{\partial t} \\
&= X_H + tE(H)\frac{\partial}{\partial t}.
\end{align*}
Hence, the Jacobi Hamiltonian dynamics on $(\Lambda,E)$ is precisely the projection of the homogeneous Poisson Hamiltonian flow associated with $\hat H$. 

In general, $H$ is not preserved along Jacobi trajectories, since $X_H(H)=-H E(H)$. Conservation occurs only if $E(H)=0$, automatically satisfied in the Poisson case. Lifting to the homogeneous Poisson system provides a geometric explanation
\[
X_{\hat H}(\hat H) = X_H(tH) + tE(H)\frac{\partial}{\partial t}(tH) = -tH E(H) + tH E(H) = 0.
\]

This shows that the energy variation observed along Jacobi trajectories corresponds to an exchange with the scale variable $t$ in the homogeneous Poisson dynamics. In other words, the apparent dissipation of $H$ arises from projecting a conservative Hamiltonian flow onto a fixed slice of the dilation variable $t$. From this viewpoint, Jacobi dynamics can be interpreted as conservative Hamiltonian motion modulo dilations, with the evolution of $H$ measuring the redistribution of energy between the original system and the scale direction generated by the Reeb vector field $E$.

Consequently, since JHIs preserve the homogeneous Hamiltonian $\hat H$, they reproduce the correct dissipation behaviour of the original Jacobi Hamiltonian flow.

\subsubsection{Homogeneity of Lagrangian bisections}

We now examine the homogeneous structure of the Lagrangian bisections that arise in the iterative construction of the integrator. According to Theorem~\ref{thm: S_i homogeneous}, the Lagrangian bisection approximating the exact homogeneous Lagrangian bisection to order $k$ is given by the embedding
\[
\tilde{L}_k : y \mapsto \Big(y, \sum_{i=1}^k \Delta s^i \nabla S_i(y) \Big),
\]
and is homogeneous. More precisely, from the proof of the theorem, the generating function satisfies
\[
\frac{\partial S_s^{(k)}}{\partial x}(x, zt) = z \frac{\partial S_s^{(k)}}{\partial x}(x,t), \qquad 
\frac{\partial S_s^{(k)}}{\partial t}(x, zt) = \frac{\partial S_s^{(k)}}{\partial t}(x,t),
\]
for all $z \in \mathbb{R}^\times$.

Expressed in coordinates $(x,t)$, the embedding reads
\[
\tilde{L}_k(x,t) = \Big(x, t, \sum_{i=1}^k \Delta s^i \nabla_x S_i(x,t), \sum_{i=1}^k \Delta s^i \nabla_t S_i(x,t) \Big),
\]
and satisfies the homogeneity condition
\begin{align*}
(\tilde{L}_k \circ h_z)(x,t) 
&= \Big(x, t, \sum_{i=1}^k \Delta s^i \nabla_x S_i(x, zt), \sum_{i=1}^k \Delta s^i \nabla_t S_i(x, zt) \Big)\\
&= \Big(x, zt, z \sum_{i=1}^k \Delta s^i \nabla_x S_i(x,t), \sum_{i=1}^k \Delta s^i \nabla_t S_i(x,t) \Big)\\
&= (T^* h_z \circ \tilde{L}_k)(x,t),
\end{align*}
showing that $\tilde{L}_k$ indeed preserves the homogeneous structure.

Consequently, the diffeomorphism generated by $\tilde{L}_k$ approximates the exact flow up to order $k$:
\[
\varphi_{\tilde{L}_k} = \phi^s_{\hat{H}} + o(\Delta s^k).
\]
This result confirms that the integrator inherits the homogeneity of the underlying Hamiltonian system and that the error in the flow vanishes faster than $(\Delta s)^k$ as the step size tends to zero.

\section{Numerical examples}\label{sec:numerical}

In this section, we illustrate the construction of Jacobi Hamiltonian Integrators for various examples. The procedure follows the workflow summarized in Figure~\ref{diagram}, which we now explain in detail.

\begin{figure}[h!]
    \centering
    \begin{tikzpicture}[node distance=2cm]
    \node (start) [startstop] {Given $(\Lambda,E)$ and Hamiltonian $H$};
    \node (Poissonization) [process, below of=start, yshift = 0.2cm] {Poissonization: $\Pi = \frac{1}{t}\Lambda + \dfrac{\partial}{\partial t}\wedge E$, $\hat{H}=tH$};
    \node (dec1) [decision, below of=Poissonization] {\begin{tabular}{c}
            Is the homogeneous symplectic\\
            bi-realization known?
        \end{tabular} };
    
    \node (pro2a) [process, below of=dec1, yshift=-1.5cm] {Compute $S_i$ until the desired order};
    \node (pro2b) [process, right of=dec1,xshift = 2cm, yshift=-1.75cm] {Compute the symplectic bi-realization};
    \node (final) [startstop, below of=pro2a, yshift = -0.3cm] {\begin{tabular}{l}
            Solve: $\alpha\left(y_n, \sum_{i=1}^k \Delta s^i \nabla S_i(y_n) \right) = x_n$\\
            Update: $x_{n+1}=\beta\left(y_n, \sum_{i=1}^k \Delta s^i \nabla S_i(y_n) \right)$
        \end{tabular}};
    
        \draw [arrow] (start) -- (Poissonization);
        \draw [arrow] (Poissonization) -- (dec1);
        \draw [arrow] (dec1) -- node[anchor=east] {yes} (pro2a);
        \draw [arrow] (dec1) -| node[anchor=south] {no} (pro2b);
        \draw [arrow] (pro2b) |- (pro2a);
        \draw [arrow] (pro2a) -- (final);
    \end{tikzpicture}
    \caption{Flowchart for the construction of Jacobi Hamiltonian Integrators (JHI).}\label{diagram}
\end{figure}
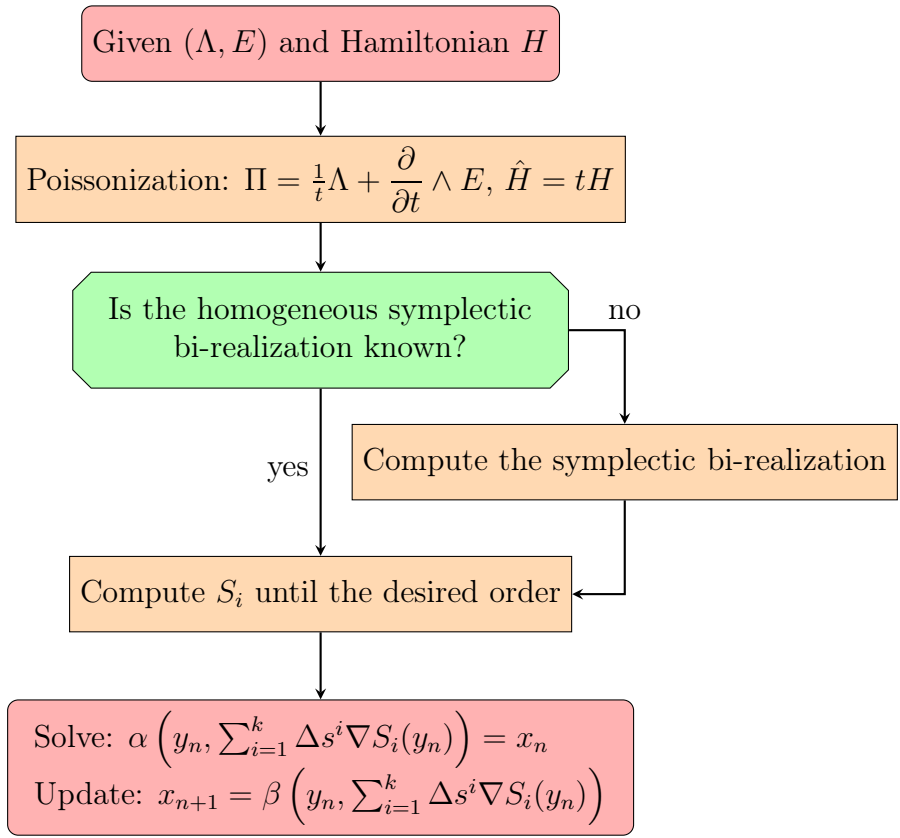

\paragraph{Step 1: Poissonization.} 
Given a Jacobi manifold $(\Lambda, E)$ and a Hamiltonian $H$, we first transform the system into a homogeneous Poisson system by defining
\[
\Pi = \frac{1}{t}\Lambda + \frac{\partial}{\partial t} \wedge E, \qquad \hat{H} = t H,
\]
where $t$ is an auxiliary variable introducing homogeneity. This step allows the use of methods developed for homogeneous Poisson systems.

\paragraph{Step 2: Symplectic bi-realization.} 
Next, we check whether a homogeneous symplectic bi-realization $(\alpha, \beta)$ of the Poisson manifold is already known. If it is not available, it must be computed first. This bi-realization is essential for lifting the Hamiltonian flow to a symplectic space, where the JHI is constructed.

\paragraph{Step 3: Computation of $S_i$.} 
Once the bi-realization is available, we compute the functions $S_i$ recursively up to the desired order $k$, as described in Section~\ref{sec: explicit construction}. These functions encode the Hamiltonian flow in the symplectic space and determine the accuracy of the integrator.

\paragraph{Step 4: Construction of the integrator.} 
Finally, the $n$-th step of the method consists of solving
\[
\alpha\Big(y_n, \sum_{i=1}^k \Delta s^i \nabla S_i(y_n)\Big) = x_n,
\]
for $y_n$, and then updating
\[
x_{n+1} = \beta\Big(y_n, \sum_{i=1}^k \Delta s^i \nabla S_i(y_n)\Big),
\]
which gives the next point of the trajectory. This step completes one iteration of the JHI.

The numerical experiments are organized into three groups.  
The first example (Section~\ref{sec: exe contact}) is devoted to comparing the qualitative behavior of JHI with that of a second-order Runge-Kutta method.  
The examples in Sections~\ref{sec: exe 2D}, \ref{sec: exe 3D}, and \ref{sec: exe 4D} are illustrative cases, designed to illustrate the order of accuracy of the constructed integrators, including the situation in which only an approximate symplectic bi-realization is available.  
Finally, the examples in Sections~\ref{sec: exe Damped}, \ref{sec: exe LV} and \ref{sec: exe RBR} correspond to classical models widely studied in the literature, allowing for a direct comparison with existing numerical approaches.

An interesting feature arises in Section~\ref{sec: exe 2D}: for the chosen Hamiltonian, the associated JHI reproduces the exact solution of the system. In this case, the discrete flow generated by the integrator coincides with the continuous Hamiltonian flow, and no numerical error is introduced

\subsection{Contact Hamiltonian system}\label{sec: exe contact}

Consider a 3-dimensional manifold with coordinates $(q,p,z)$. The canonical contact structure \cite{Simoes_2020} is defined by
\begin{align}\label{eq: contact bivetor field}
    \Lambda = \frac{\partial}{\partial p}\wedge \frac{\partial}{\partial q} + p\,\frac{\partial}{\partial p}\wedge \frac{\partial}{\partial z}, \qquad 
    E = -\frac{\partial}{\partial z}.
\end{align}

For a Hamiltonian $H$, the associated Hamiltonian vector field is
\begin{align*}
    X_H = -\frac{\partial H}{\partial p}\frac{\partial}{\partial q} 
          + \left(\frac{\partial H}{\partial q} + p \frac{\partial H}{\partial z}\right)\frac{\partial}{\partial p} 
          + \left(H - p \frac{\partial H}{\partial p}\right)\frac{\partial}{\partial z}.
\end{align*}

The associated Poisson (symplectic) structure is 
\begin{align}\label{eq: poisson - contact}
    \Pi = \frac{1}{t}\frac{\partial}{\partial p}\wedge \frac{\partial}{\partial q} + \frac{p}{t}\frac{\partial}{\partial p}\wedge \frac{\partial}{\partial z} - \frac{\partial}{\partial t}\wedge \frac{\partial}{\partial z}.
\end{align}

\paragraph{Construction of a first-order JHI.} Following the flowchart in Figure~\ref{diagram}, the first step is to \emph{Poissonize} the Jacobi manifold. Here, the original contact manifold $(\Lambda,E)$ and Hamiltonian $H$ are lifted to the homogeneous Poisson manifold $(\Pi, \hat{H})$ with 
\[
\Pi = \frac{1}{t}\Lambda + \frac{\partial}{\partial t}\wedge E, \qquad \hat{H} = t H.
\]

Next, the diagram asks whether a homogeneous symplectic bi-realization $(\alpha,\beta)$ is known. In this example, constructing it directly through Poisson spray is difficult due to the $p$-dependence in \eqref{eq: poisson - contact}. To overcome this, we introduce the symplectomorphism
\[
F(q,p,z,t) = (Q,P,Z,T) = (q, pt, -z, t),
\]
which transforms $\Pi$ into the canonical symplectic structure
\[
F_*\Pi = \frac{\partial}{\partial P}\wedge \frac{\partial}{\partial Q} + \frac{\partial}{\partial T}\wedge \frac{\partial}{\partial Z}.
\]

Using the known canonical bi-realization \cite{cosserat2023numerical}, we transport it back to the original coordinates via $F$. This is illustrated by the diagram:
\begin{center}
\begin{tikzpicture}
    \node at (0,0) (a) {$(M_1,\Pi)$};
    \node at (0,-2) (b) {$T^*M_1$};
    \node at (3,0) (c) {$(M_2,\omega_\text{can}^{-1})$};
    \node at (3,-2) (d) {$T^*M_2$};
    \draw [->] (b) to node[left] {$\alpha$} (a);
    \draw [->] (a) to node[above] {$F$} (c);
    \draw [->] (d) to node[right] {$\alpha_\text{can}$} (c);
    \draw [->] (b) to node[above] {$(F^{-1})^*$} (d);
\end{tikzpicture}
\end{center}

This yields the explicit formulas:
\begin{equation}\label{eq: alpha contact}
\begin{aligned}
\alpha(x,\xi_x) &= F^{-1}\circ \alpha_\mathrm{can} \circ (F^{-1})^* (x,\xi_x) 
= \left(q-\frac{\xi_p}{2t},\ \frac{tp+\frac{\xi_q}{2}}{t-\frac{\xi_z}{2}}, z+\frac{\xi_t}{2} - \frac{p\xi_p}{2t}, t-\frac{\xi_z}{2}\right), \\
  \ \beta(x,\xi_x) &= \alpha(x,-\xi_x).
\end{aligned}
\end{equation}

Following the next steps of the flowchart, once $(\alpha,\beta)$ is known, we compute the first Taylor coefficient $S_1(x) = \hat{H}(x)$. Then the first-order Jacobi Hamiltonian Integrator is
\begin{enumerate}
    \item Solve for $y_n$:
    \[
        \alpha\big(y_n, \Delta s \nabla S_1(y_n)\big) = x_n,
    \]
    \item Update the solution:
    \[
        x_{n+1} = \beta\big(y_n, \Delta s \nabla S_1(y_n)\big),
    \]
\end{enumerate}
where $\Delta s$ is the time step.

\paragraph{Numerical example.} 
Consider the Hamiltonian $H(q,p,z) = q + z$ with initial condition $x_0 = (0.1, -1.1, 0.09)$. Over the time interval $[0,20]$, the $p$ and $z$ coordinates grow rapidly. Figure~\ref{fig: contact singularity} compares the first-order JHI (using both exact and approximate bi-realizations) with a classical second-order Runge--Kutta method ($\Delta s = 0.1$). Despite being first-order, the JHI delays the blow-up compared to RK-2, illustrating the benefit of geometric structure preservation.
\begin{figure}[h!]
\centering
\includegraphics[width=1\textwidth]{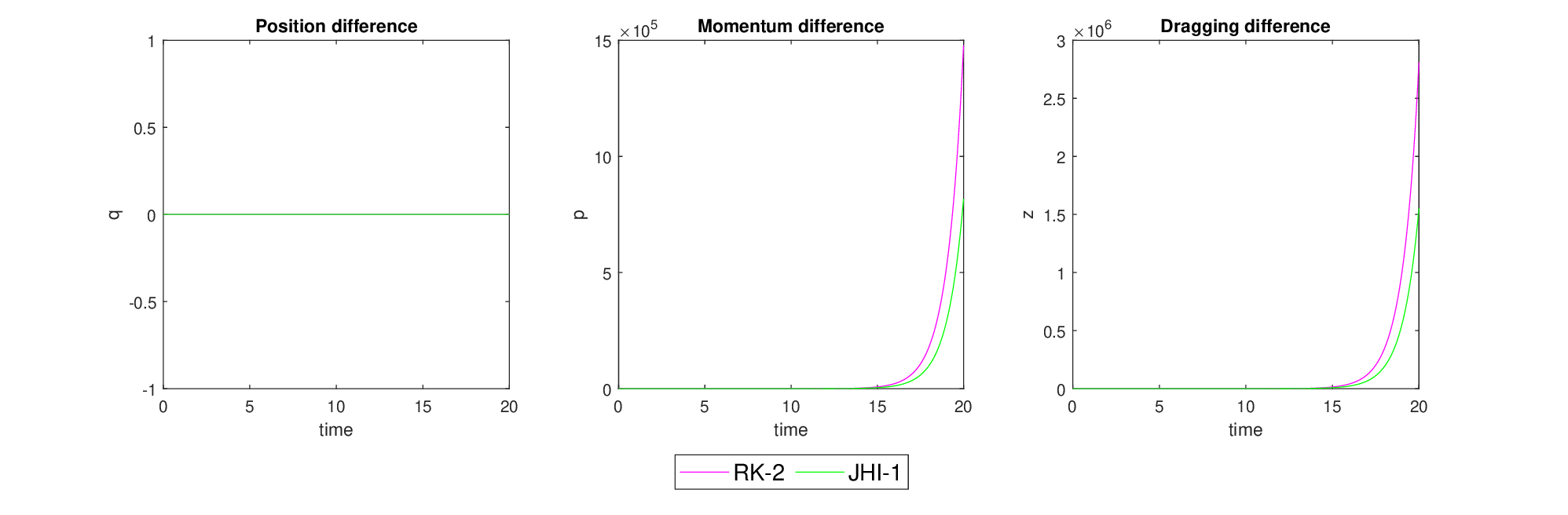}
\caption{Comparison between JHI and RK-2 with the exact solution for the contact Hamiltonian system.}
\label{fig: contact singularity}
\end{figure}

\subsection{A two-dimensional Jacobi Hamiltonian problem}\label{sec: exe 2D}

We consider a 2D Jacobi structure on $\mathbb{R}^2$
\[
\Lambda = (x^2+y^2)\frac{\partial}{\partial x}\wedge \frac{\partial}{\partial y}, \qquad 
E = 2x\frac{\partial}{\partial y} - 2y\frac{\partial}{\partial x}.
\]

The associated Hamiltonian vector field for a Hamiltonian $H(x,y)$ is
\[
X_H = \left((x^2+y^2)\frac{\partial H}{\partial y}+2yH\right)\frac{\partial}{\partial x} - \left((x^2+y^2)\frac{\partial H}{\partial x}+2xH\right)\frac{\partial}{\partial y}.
\]

\begin{remark}
    In dimension 2 any bivector $\Lambda$ automatically satisfies $[\Lambda, \Lambda] = 0 = 2 E \wedge \Lambda$ and together with the equation $[E, \Lambda] = 0$, this means that any 2D Jacobi structure $(\Lambda, E)$ amounts to a Poisson structure $\Lambda$ together with a Poisson vector field $E$.
\end{remark}

\paragraph{Construction of the JHI.}
Following the flowchart in Figure~\ref{diagram}, the homogeneous Poisson structure is
\[
\Pi = \frac{x^2+y^2}{t}\frac{\partial}{\partial x}\wedge \frac{\partial}{\partial y} + 2x\frac{\partial}{\partial t}\wedge \frac{\partial}{\partial y} - 2y\frac{\partial}{\partial t}\wedge \frac{\partial}{\partial x}, 
\qquad \hat{H} = t H(x,y).
\]

A direct construction of a symplectic bi-realization using a Poisson spray \cite{CrainicMarcut2011} is technically cumbersome in this case, due to the nonlinearity and degeneracy of the coefficients. Instead, we exploit the local normal form provided by Weinstein's splitting theorem \cite{Weinstein1983}, which offers a more transparent and geometrically natural approach. 

To this end, we introduce polar coordinates
\[
x=r\cos\theta,\qquad y=r\sin\theta,
\]
so that $r^2 = x^2 + y^2$. In these coordinates, the Poisson structure takes the form 
\[
\Pi=\frac{r^2}{t}\,\partial_r\wedge\partial_\theta
+2r\,\partial_t\wedge\partial_\theta .
\]

A direct computation shows that \(c=r^2/t\) is a Casimir function of $\Pi$. Defining the change of coordinates
\[
F(x,y,t)=(q,p,c)=\left(\theta,-\tfrac{t}{2},\tfrac{r^2}{t}\right)
\]
we obtain the canonical Poisson structure
\[
F_*\Pi=\Pi_{\mathrm{can}}=\partial_q\wedge\partial_p,
\]
with $c$ as a transverse (Casimir) variable.

For \(\Pi_{\mathrm{can}}\), a symplectic bi-realization is given by
\[
\alpha_{\mathrm{can}}(q,p,c,\xi_q,\xi_p,\xi_c)
=
\left(q-\frac{1}{2}\xi_p,\; p+\frac{1}{2}\xi_q,\; c\right),
\quad
\beta_{\mathrm{can}}(q,p,c,\xi_q,\xi_p,\xi_c)
=
\left(q+\frac{1}{2}\xi_p,\; p-\frac{1}{2}\xi_q,\; c\right).
\]

Transporting it back yields the bi-realization
\[
\alpha(X,\xi)=(F^{-1}\circ\alpha_{\mathrm{can}}\circ (F^{-1})^*)(X,\xi),
\qquad
\beta(X,\xi)=\alpha(X,-\xi),
\]
associated with the original Poisson structure $\Pi$.

\paragraph{Numerical examples.}  
We first compute trajectories for the Hamiltonian
\[
H(x,y)=x^2+y^2
\]
with initial condition \(x_0=(1,1)\), time step \(\Delta s=0.03\), and time interval \([0,\pi]\).
As shown in Figure~\ref{fig:Jacobi2D}, the JHI trajectory coincides with the exact solution and exhibits better qualitative behavior than the second-order Runge--Kutta (RK-2) method.

\begin{figure}[!ht]
\centering
\includegraphics[width=0.7\textwidth]{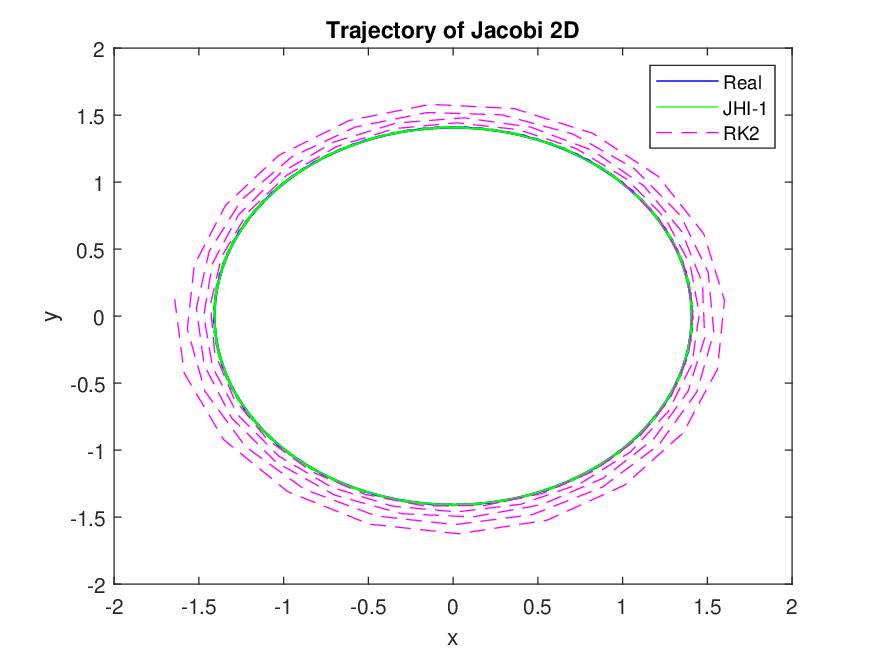}
\caption{Comparison between the exact trajectory and those obtained by the JHI and RK-2 methods for the Hamiltonian \(H=x^2+y^2\).}
\label{fig:Jacobi2D}
\end{figure}

In fact, this Hamiltonian leads to an exact JHI scheme. Indeed, since
\[
\hat H(\alpha(dS^{(1)}))=t(x^2+y^2)
\]
is independent of the parameter \(s\), we have
\[
S_2=\frac{1}{2}\left.\frac{d}{ds}\right|_{s=0}\hat{H}(\alpha(dS^{(1)}))=0,
\]
and recursively \(S_{k}=0\) for all \(k\ge2\). Consequently, the JHI reproduces the exact flow of the Hamiltonian vector field.

We now consider a non-quadratic Hamiltonian,
\[
H(x,y)=\cos(x)\sin(y),
\]
and compute trajectories on the interval \([0,\pi]\) with time step \(\Delta s=0.1\) and initial condition \(x_0=(1,1)\).
Figure~\ref{fig:Jacobi2D2} shows that the first-order JHI still provides a more accurate qualitative approximation than RK-2.

\begin{figure}[!ht]
	\centering
	\includegraphics[width=0.7\textwidth]{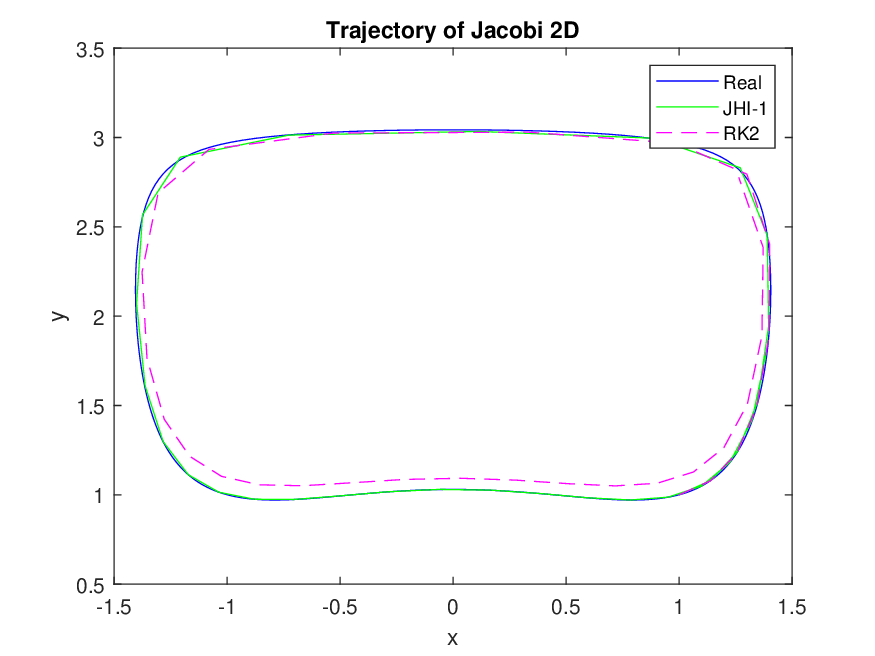}
	\caption{Comparison between the exact trajectory and those obtained by the JHI and RK-2 methods for the Hamiltonian \(H=\cos(x)\sin(y)\).}
	\label{fig:Jacobi2D2}
\end{figure}

To estimate the numerical order of the first-order JHI for this Hamiltonian, we compute solutions on \([0,1]\) with initial condition \(x_0=(1,1)\) using time grids with
\[
n=2^{k+1}+1,\qquad k=1,\dots,9.
\]

The reference solution is obtained using a high-order method with \(2^{11}+1\) time steps. Errors are computed by comparing
\[
x^{(k)}(i)-x^\ast\!\left(2^{10-k}(i-1)+1\right).
\]

The results in Table~\ref{Tab:2Dorder} indicate a numerical convergence rate close to second order.
\begin{table}[!ht]
\centering
\begin{tabular}{c c c}
\hline
$\Delta s$ & $\| \cdot \|_2$ error & Order \\ \hline
$2^{-2}$ & $1.61\times10^{-2}$ & -- \\
$2^{-3}$ & $4.80\times10^{-3}$ & $1.74$ \\
$2^{-4}$ & $1.30\times10^{-3}$ & $1.89$ \\
$2^{-5}$ & $3.30\times10^{-4}$ & $1.95$ \\
$2^{-6}$ & $8.50\times10^{-5}$ & $1.97$ \\
$2^{-7}$ & $2.10\times10^{-5}$ & $1.99$ \\
$2^{-8}$ & $5.40\times10^{-6}$ & $1.99$ \\
$2^{-9}$ & $1.30\times10^{-6}$ & $2.00$ \\ \hline
\end{tabular}
\caption{Error and numerical order of the first-order JHI for the 2D Jacobi Hamiltonian problem.}
\label{Tab:2Dorder}
\end{table}

In the present system, the evolution of the Hamiltonian along the flow is not $0$,
\[
E(H) = 2x\cos(x)\cos(y) + 2y\sin(x)\sin(y) \neq 0,
\] 
so the Hamiltonian is not preserved by the Jacobi flow. 
This precludes a direct comparison of $H$ along trajectories with contour values without knowledge of the exact solution.

To address this, we consider the lifted Poisson Hamiltonian
\[
\hat{H}(x,t) = t H(x),
\] 
which is conserved. 
The deviation along a trajectory can then be measured by
\[
\frac{\hat{H}(x_0,t_0) - \hat{H}(x,t)}{t}.
\] 

As shown in Figure~\ref{fig:Hamil2D}, the first-order Jacobi Hamiltonian integrator (JHI) exhibits an error of order $10^{-5}$, 
whereas the second-order Runge--Kutta method (RK-2) shows an error of order $10^{-3}$.

\begin{figure}[!ht]
\centering
\includegraphics[width=0.7\textwidth]{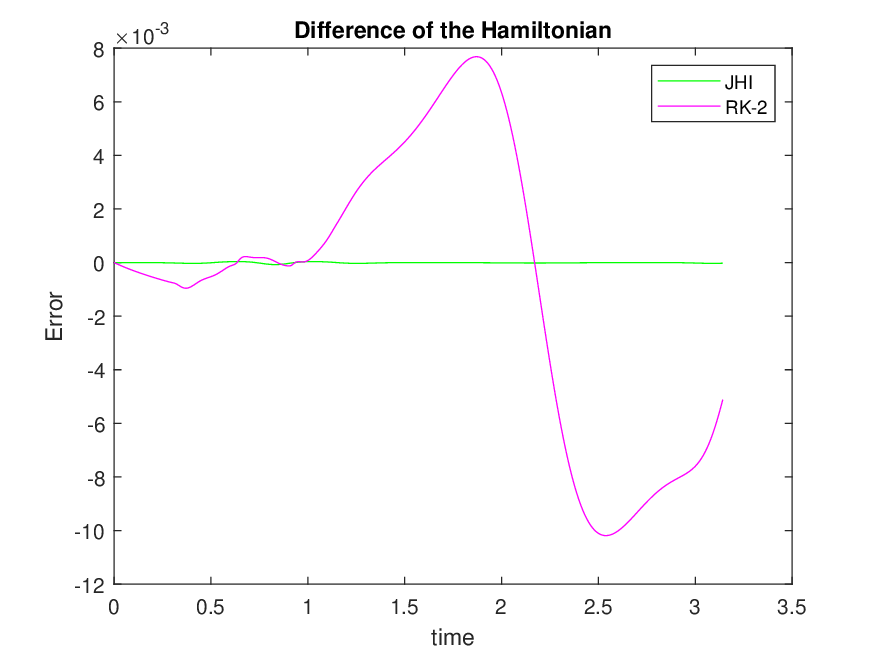}
\caption{Hamiltonian error along trajectories computed with JHI and RK-2.}
\label{fig:Hamil2D}
\end{figure}

\subsection{A three-dimensional Jacobi Hamiltonian problem}\label{sec: exe 3D}

Consider the following linear Jacobi structure
\begin{align*}
    \Lambda = 2x_2\frac{\partial}{\partial x_1}\wedge\frac{\partial}{\partial x_2} + 2x_3 \frac{\partial}{\partial x_1}\wedge \frac{\partial}{\partial x_3}, \qquad E = \frac{\partial}{\partial x_1},
\end{align*}
and the Hamiltonian $H(x) = x_1^2+x_2^2 + x_3^2$. The Hamiltonian vector field $X_H$ is
\begin{align*}
    X_H &= \Lambda(\cdot,dH) - HE(\cdot)\\
        &= (2x_2\partial_{x_2}H + 2x_3\partial_{x_3}H - H)\frac{\partial}{\partial x_1} - 2x_2\partial_{x_1}H \frac{\partial}{\partial x_2} - 2x_3\partial_{x_1}H \frac{\partial}{\partial x_3}.
\end{align*}

The homogeneous Poisson structure associated is
\begin{align}\label{eq: Pi J3D}
    \Pi = \frac{2x_2}{t}\frac{\partial}{\partial x_1}\wedge\frac{\partial}{\partial x_2} +\frac{2x_3}{t}\frac{\partial}{\partial x_1}\wedge\frac{\partial}{\partial x_3} + \frac{\partial}{\partial t}\wedge\frac{\partial}{\partial x_1}.
\end{align}

We are going to use the same technique as in the previous section to compute the symplectic bi-realization. From (\ref{eq: Pi J3D}), it is easy to check that $C_2 = t^2x_2$ and $C_3 = t^2x_3$ are Casimir functions of the given Poisson structure. Consider the transformation $F(x_1,x_2,x_3,t) = (x_1,t^2x_2, t^2x_3, t) = (X_1,X_2,X_3,X_4)$. If we compute the pushforward of $\Pi$ through $F$ we get $F_*\Pi = \dfrac{\partial}{\partial X_4}\wedge \dfrac{\partial}{\partial X_1}$, and now it is easy to compute the symplectic bi-realization in this case, which is
\begin{align*}
	\alpha_\text{can}(X,\eta) = (X_1+\frac{1}{2}\eta_4,X_2,X_3,X_4-\frac{1}{2}\eta_1)
\end{align*}
and $\beta_\text{can}(X,\eta) = \alpha(X,-\eta)$. So, the symplectic bi-realization on the original manifold is computed as $\alpha = F^{-1}\circ \alpha_\text{can}\circ (F^{-1})^*$ which is
\begin{align*}
	\alpha(x,t,\xi_x,\xi_t) = \left(x_1+\frac{1}{2}\xi_t - \frac{x_2\xi_2+x_3\xi_3}{t},\frac{t^2x_2}{(t-\frac{1}{2}\xi_1)^2}, \frac{t^2x_3}{(t-\frac{1}{2}\xi_1)^2}, t-\frac{1}{2}\xi_1 \right)
\end{align*}
and $\beta(x,t,\xi_x,\xi_t) = \alpha(x,t,-\xi_x,-\xi_t)$.

Computing the numerical method of higher order, we get $S_2=0$ and 
\begin{align*}
	S_3(x) = t\frac{3x_3^4- x_1^4 + 10x_1^2x_2^2 + 10x_1^2x_3^2 + 3x_2^4 + 6x_2^2x_3^2 }{4}.
\end{align*}

The first-order JHI (JHI-1) is defined by taking only $S_1$ in the expansion, while the third-order JHI (JHI-3) includes both $S_1$ and $S_3$, producing a higher-order accurate method. 

\paragraph{Numerical examples.}  
We compute the trajectories for $H(x) = x_1^2 + x_2^2+x_3^2$ with initial condition $x_0 = [-1,1,1]$, time 
time step $\Delta s = 0.1$, and  time span $[0,2]$. Figure~\ref{fig:Jacobi3D} shows that first- and third-order JHI trajectories approximate the exact flow much better  than the RK-2. 
\begin{figure}[!ht]
\centering
\includegraphics[width=11cm]{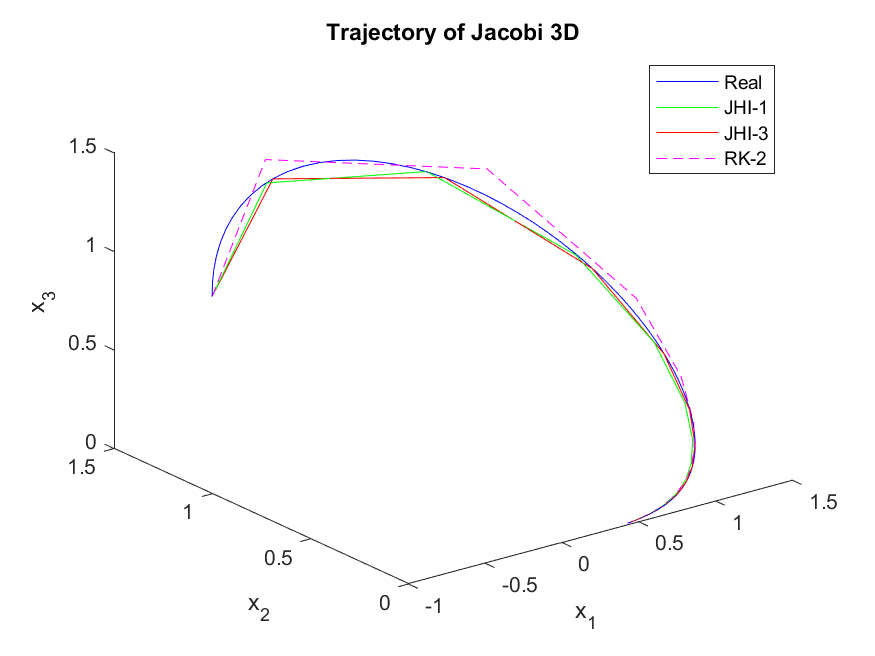}
\caption{Comparison between the actual trajectory and the JHI (first and third order) and RK-2 methods for a 3D problem.}
\label{fig:Jacobi3D}
\end{figure}

The numerical order is estimated on $[0, 0.9]$ by varying $\Delta s$. Table~\ref{Tab:3Dorder} shows 
that, for this example, first-order JHI-1 has order  $2$ and third-order JHI-3 has order $4$.

\begin{table}[!ht]
\centering
\begin{tabular}{c c c c c}
\hline
$\Delta s$ & $\|\cdot\|_2$ JHI-1 error & Order JHI-1 & $\|\cdot\|_2$ JHI-3 error & Order JHI-3 \\ \hline
0.225   & $1.99 \times 10^{-1}$  & --   & $6.58 \times 10^{-2}$  & --   \\
0.1125  & $6.25 \times 10^{-2}$  & 1.68 & $8.70 \times 10^{-3}$  & 2.91 \\
0.0563  & $1.68 \times 10^{-2}$  & 1.89 & $6.50 \times 10^{-4}$  & 3.75 \\
0.0281  & $4.30 \times 10^{-3}$  & 1.97 & $4.20 \times 10^{-5}$  & 3.93 \\
0.0141  & $1.10 \times 10^{-3}$  & 1.99 & $2.70 \times 10^{-6}$  & 3.98 \\
0.0070  & $2.70 \times 10^{-4}$  & 1.99 & $1.60 \times 10^{-7}$  & 3.99 \\
0.0035  & $6.70 \times 10^{-5}$  & 2.00 & $1.00 \times 10^{-8}$  & 3.99 \\
0.0018  & $1.60 \times 10^{-5}$  & 2.00 & $6.60 \times 10^{-10}$ & 4.00 \\ \hline
\end{tabular}
\caption{Error and numerical order of first- and third-order JHI for the 3D Jacobi Hamiltonian problem.}
\label{Tab:3Dorder}
\end{table}

Although $H$ is not conserved along the Jacobi flow ($X_H(H)=-2x_1 H$), the lifted Hamiltonian $H_p=tH$ is preserved. As shown in Figure~\ref{fig:Hamil3D} for the simulation over $[0,10]$ with step size $\Delta s = 0.01$, the first- and third-order JHI methods exhibit  superior Hamiltonian preservation compared with RK-2. 

\begin{figure}[!ht]
\centering
\includegraphics[width=0.7\textwidth]{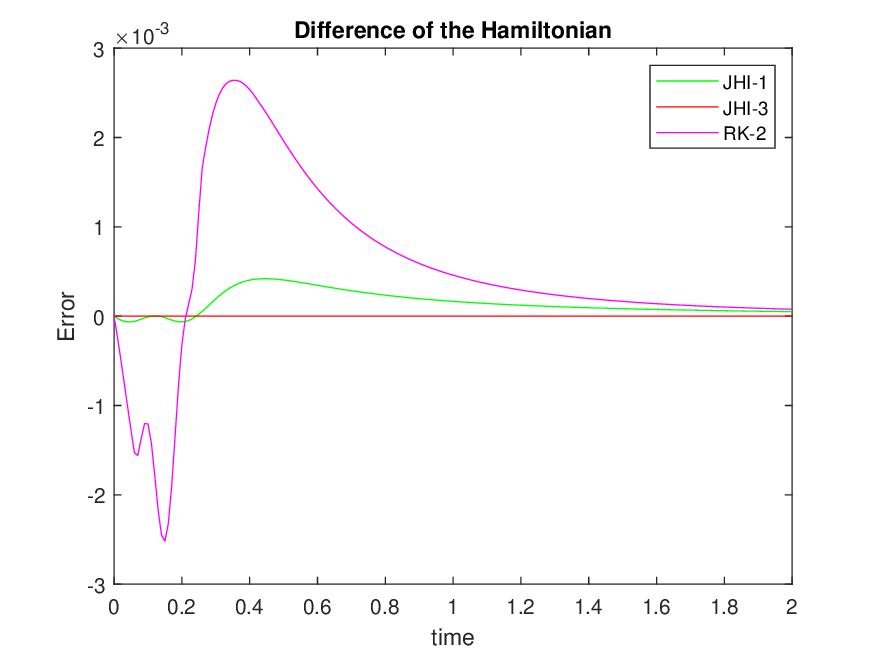}
\caption{Hamiltonian difference along 3D trajectories: JHI (1st and 3rd order) vs. RK-2.}
\label{fig:Hamil3D}
\end{figure}

Figure~\ref{fig:casimir3D} demonstrates that the Casimir functions 
$C_2 = t^2 x_2$ and $C_3 = t^2 x_3$ are preserved by both first- and third-order JHI schemes up to machine precision, with errors on the order of $10^{-14}$.

\begin{figure}[!ht]
\centering
\includegraphics[width=1\textwidth]{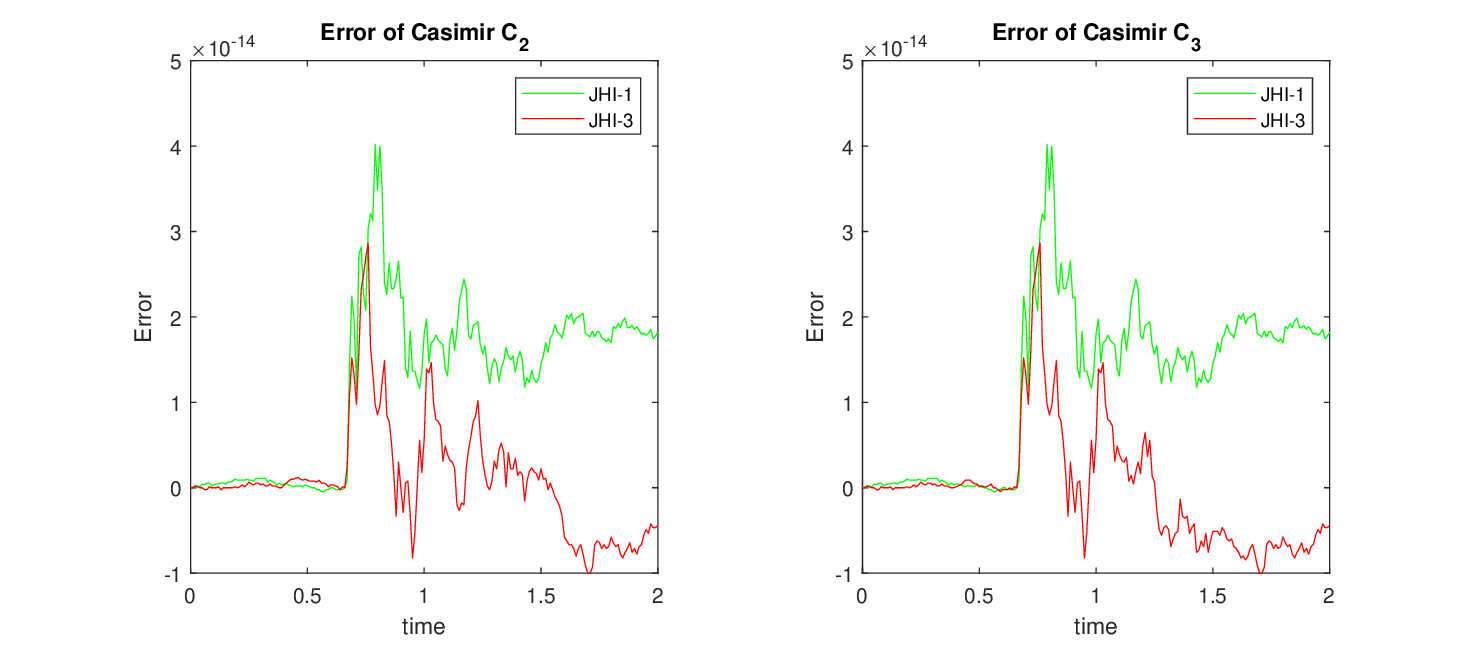}
\caption{Casimir preservation along 3D trajectories: JHI (1st and 3rd order).}
\label{fig:casimir3D}
\end{figure}

\subsection{A four-dimensional Jacobi Hamiltonian problem}\label{sec: exe 4D}

We consider a four-dimensional Jacobi structure defined by
\[
\Lambda = \cos(x_2)\frac{\partial}{\partial x_1}\wedge\frac{\partial}{\partial y_1}, \qquad
E = e^{y_2}\left(y_1\frac{\partial}{\partial x_1}+x_1\frac{\partial}{\partial y_1} \right),
\]
with $(x_1,y_1,x_2,y_2)\in\mathbb{R}^4$. Given a Hamiltonian function $H(x,y)$, the associated Hamiltonian vector field is
\[
X_H = \big(\cos(x_2)\partial_{y_1}H - e^{y_2}y_1H\big)\frac{\partial}{\partial x_1} - \big(\cos(x_2)\partial_{x_1}H + e^{y_2}x_1H\big)\frac{\partial}{\partial y_1}.
\]

To construct a Jacobi Hamiltonian Integrator, we first lift this Jacobi structure to a homogeneous Poisson structure:
\[
\Pi = \frac{\cos(x_2)}{t}\frac{\partial}{\partial x_1}\wedge\frac{\partial}{\partial y_1} + e^{y_2}\frac{\partial}{\partial t}\wedge\left(y_1\frac{\partial}{\partial x_1}+x_1\frac{\partial}{\partial y_1} \right).
\]

The null-space of $\Pi$ is easily computed:
\[
\mathcal{N}(\Pi) = \left\{
\begin{bmatrix}0\\0\\1\\0\\0\end{bmatrix},
\begin{bmatrix}0\\0\\0\\1\\0\end{bmatrix},
\begin{bmatrix}-tx_1 e^{y_2}/\cos(x_2)\\ t y_1 e^{y_2}/\cos(x_2)\\0\\0\\1\end{bmatrix}
\right\}.
\]

From this, we see that a Casimir function $C$ must have a gradient collinear with these vectors. Scaling the third vector by $\cos(x_2)/t$ and adding multiples of the first two vectors leads to
\[
\nabla C = (-x_1 e^{y_2}, y_1 e^{y_2}, f, g, \cos(x_2)/t),
\]
where $f$ and $g$ can be chosen freely. A natural choice is
\[
C(x,t) = \cos(x_2)\log t - \frac{e^{y_2}}{2}(x_1^2-y_1^2), \qquad f=\partial_{x_2}C, \ g=\partial_{y_2}C.
\]

Because finding a canonical transformation for this system is difficult, we construct an approximate first-order homogeneous bi-realization (as in \cite{cabrera2024}). This provides the basis for the Jacobi Hamiltonian Integrator of order 1 (JHI-1), and is given by
\begin{align*}
\alpha^{x_1}(x,t,\xi_x,\xi_t) &= x_1 + \frac{1}{2}\Big(-\frac{\cos(x_2)}{t}\xi_{y_1} + y_1 e^{y_2} \xi_t \Big),\\
\alpha^{y_1}(x,t,\xi_x,\xi_t) &= y_1 + \frac{1}{2}\Big(\frac{\cos(x_2)}{t}\xi_{x_1} + x_1 e^{y_2} \xi_t \Big),\\
\alpha^{x_2}(x,t,\xi_x,\xi_t) &= x_2, \quad \alpha^{y_2} = y_2,\\
\alpha^t(x,t,\xi_x,\xi_t) &= t - \big(y_1 e^{y_2} \xi_{x_1} + x_1 e^{y_2} \xi_{y_1}\big),
\end{align*}
with $\beta(x,t,\xi_x,\xi_t) = \alpha(x,t,-\xi_x,-\xi_t)$. This construction naturally defines the numerical method: JHI-1 uses this first-order bi-realization, while higher-order JHI methods (e.g., JHI-3) are obtained by computing additional terms in the recursive formula $S_k$ as in the 3D case. 

\paragraph{Numerical examples.} 
We take the Hamiltonian 
\[
H(x,y) = \cos(x_1) + \sin(y_1) + x_2 + y_2,
\]
and compute trajectories of the exact and numerical flows. The JHI trajectory is computed on $[0,4\pi]$ with time step $\Delta s = 0.001$ and initial condition $x_0 = [1, 1, \pi/4, -0.2]$. Figure~\ref{fig:Jacobi4D} shows excellent agreement between the JHI trajectory and the exact solution.

\begin{figure}[!ht]
\centering
\includegraphics[width=11 cm]{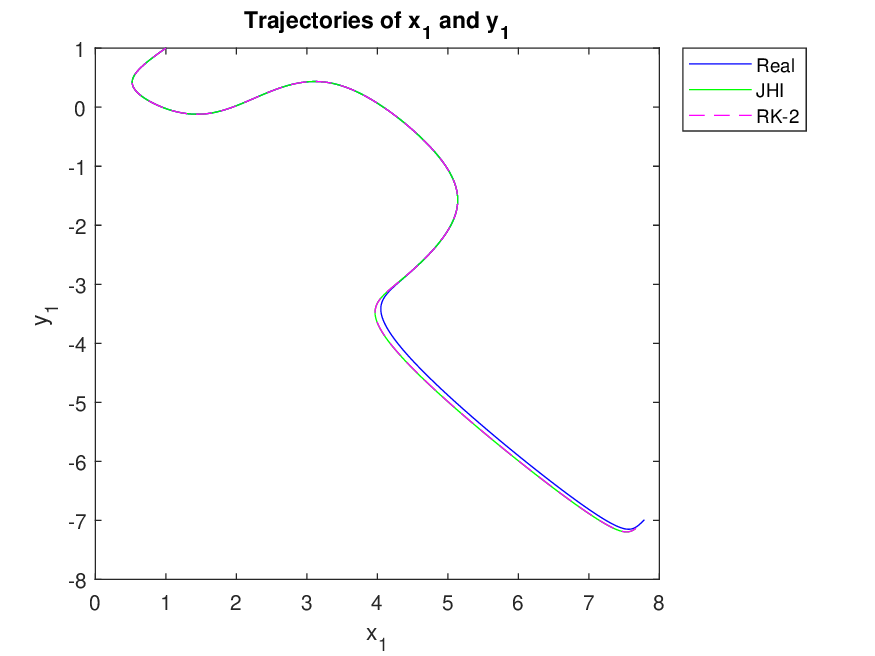}
\caption{Comparison between the exact trajectory and the JHI and RK-2 methods for the 4D Hamiltonian problem.}
\label{fig:Jacobi4D}
\end{figure}

To estimate the numerical order of the first-order JHI, we consider the time interval $[0,\pi]$ with the same initial condition. Table~\ref{Tab:4Dorder} shows that the method converges with order approximately $2$, despite the bi-realization being only first-order accurate.

\begin{table}[!ht]
\centering
\begin{tabular}{c c c}
\hline
$\Delta s$ & $\|\cdot\|_2$ error & Order \\ \hline
0.1963 & $1.49 \times 10^{-1}$ & -- \\
0.0982 & $3.71 \times 10^{-3}$ & 2.00 \\
0.0491 & $9.40 \times 10^{-4}$ & 1.98 \\
0.0245 & $2.40 \times 10^{-4}$ & 1.99 \\
0.0123 & $5.90 \times 10^{-5}$ & 1.99 \\
0.0061 & $1.40 \times 10^{-5}$ & 2.00 \\
0.0031 & $3.70 \times 10^{-6}$ & 2.00 \\
0.0015 & $9.30 \times 10^{-7}$ & 2.00 \\ \hline
\end{tabular}
\caption{Error and numerical order of first-order JHI for the 4D Jacobi Hamiltonian problem.}
\label{Tab:4Dorder}
\end{table}

Although $H$ is not preserved along the flow ($E(H)=x_1 e^{y_2}\cos(y_1)-y_1 e^{y_2}\sin(x_1)\neq0$), the JHI trajectory shows small Hamiltonian and Casimir errors (of order $10^{-8}$ and $10^{-6}$, respectively, for $[0,3\pi]$ with $\Delta s=0.001$), see Figure~\ref{fig:Hamil4D}.

\begin{figure}[!ht]
\centering
\includegraphics[width=\textwidth]{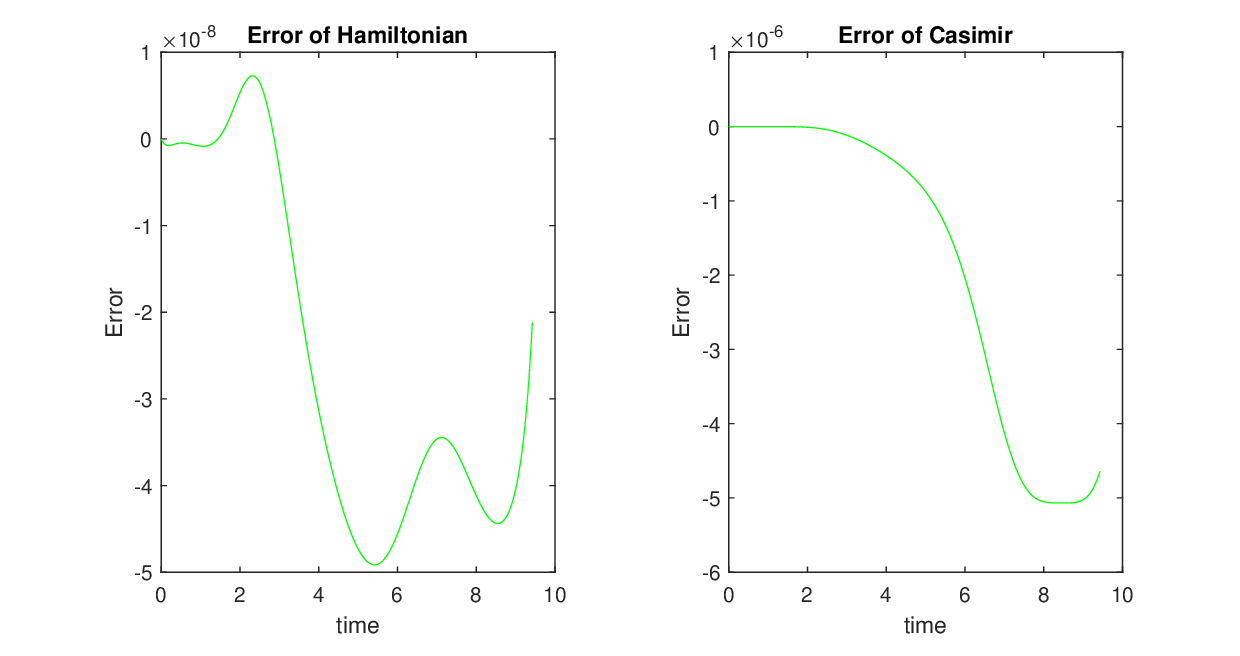}
\caption{Differences of Hamiltonian and Casimir functions along the numerical trajectories computed with JHI.}
\label{fig:Hamil4D}
\end{figure}

\subsection{Damped harmonic oscillator}\label{sec: exe Damped}
The damped harmonic oscillator is a classical example of a contact Hamiltonian system. It allows us to illustrate the construction of JHI for systems with non-conservative dynamics, where the Hamiltonian is not preserved.  

Consider the contact structure in (\ref{eq: contact bivetor field}) and the Hamiltonian
\[
H(q,p,z) = \frac{p^2}{2} + \frac{q^2}{2} + \gamma z, \qquad \gamma \in \mathbb{R},
\]
which models a damped parametric oscillator. The homogeneous Hamiltonian is 
\[
\hat{H}(q,p,z,t) = t H(q,p,z),
\] 
using the homogeneous symplectic structure as in (\ref{eq: poisson - contact}) and the symplectic bi-realization from (\ref{eq: alpha contact}).

Following the procedure in section \ref{sec: explicit construction}, we get
\begin{align*}
	\hat{H}(\alpha(dS^{(1)})) = t\left(1-\frac{s\gamma}{2}\right)\left(\frac{1}{2}(q-0.5sp)^2 + \frac{1}{2}\left(\frac{p+0.5sq}{1-0.5s\gamma}\right)^2 + \gamma(z+0.25sq^2+0.5s\gamma z-0.25sp^2)\right)
\end{align*}

For higher-order methods, $S_2 = 0$ and
\[
S_3 = \frac{1}{3!}\Bigg[\frac{q^2t}{4}(1-\gamma^2) -\frac{\gamma^3 tz}{2} + p^2 t\left(\frac{\gamma^2}{2}+\frac{1}{4}\right) + \gamma qpt \Bigg].
\]  

The corresponding JHI scheme is:
\begin{enumerate}
	\item find $y_n$ such that
	\begin{align*}
		\alpha(y_n, \Delta s \nabla \hat{H}(y_n)+\Delta s^3\nabla S_3(y_n)) = x_n
	\end{align*}
	\item calculate 
	\begin{align*}
		x_{n+1} = \beta(y_n, \Delta s \nabla \hat{H}(y_n)+\Delta s^3\nabla S_3(y_n)).
	\end{align*}
\end{enumerate}

\paragraph{Numerical examples.} 
Using time span $[0,10]$, step $\Delta s = 0.5$, initial condition $x_0=(1,0,0)$, and $\gamma = 0.01$, Figure~\ref{fig:contactdamped} shows that JHI-1 outperforms the semi-implicit symplectic Euler method, while JHI-3 nearly overlaps the exact solution.  

\begin{figure}[!ht]
\centering
\includegraphics[width=18cm]{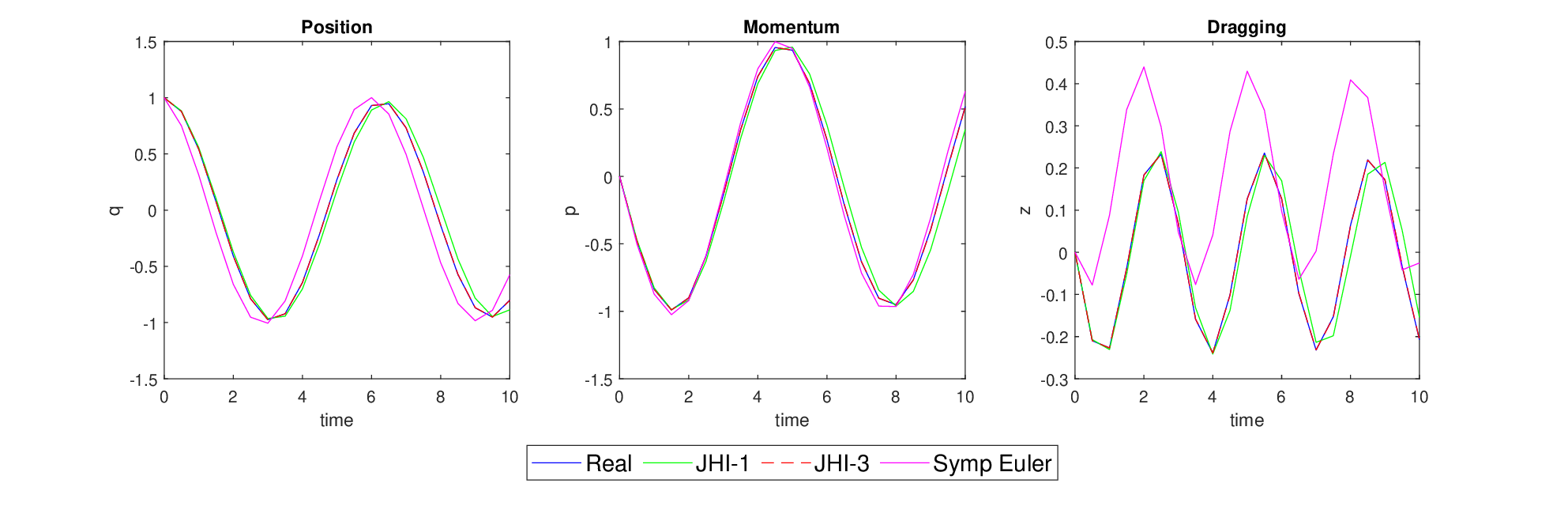}
\caption{Trajectories of the damped parametric oscillator using JHI (first and third order) and symplectic Euler methods.}
\label{fig:contactdamped}
\end{figure}

The numerical order is reported in Table \ref{Tab:contactorder}. As expected, JHI-1 and JHI-3 exhibit convergence near order 2 and 4, respectively.

\begin{table}[!ht]
\centering
\begin{tabular}{c c c c c}
\hline
$\Delta s$ & $\|\cdot\|_2$ JHI-1 error & Order JHI-1 & $\|\cdot\|_2$ JHI-3 error & Order JHI-3 \\ \hline
$2^{-1}$ & $3.45 \times 10^{-2}$ & --    & $9.30 \times 10^{-4}$  & --    \\
$2^{-2}$ & $9.90 \times 10^{-3}$ & 1.80  & $6.60 \times 10^{-5}$  & 3.81  \\
$2^{-3}$ & $2.60 \times 10^{-3}$ & 1.92  & $4.40 \times 10^{-6}$  & 3.92  \\
$2^{-4}$ & $6.60 \times 10^{-4}$ & 1.97  & $2.80 \times 10^{-7}$  & 3.97  \\
$2^{-5}$ & $1.60 \times 10^{-4}$ & 1.98  & $1.80 \times 10^{-8}$  & 3.98  \\
$2^{-6}$ & $4.20 \times 10^{-5}$ & 1.99  & $1.10 \times 10^{-9}$  & 3.99  \\
$2^{-7}$ & $1.00 \times 10^{-5}$ & 2.00  & $7.00 \times 10^{-11}$ & 4.00  \\ \hline
\end{tabular}
\caption{Error and numerical order of first- and third-order JHI for the damped contact Hamiltonian problem.}
\label{Tab:contactorder}
\end{table}

Since $E(H)=-\gamma$, the Hamiltonian is not preserved. Figure \ref{fig:dampedhamiltonian} shows that the deviation oscillates around zero with magnitude $\mathcal{O}(10^{-4})$ for JHI-1 and $\mathcal{O}(10^{-5})$ for JHI-3.

\begin{figure}[!ht]
\centering
\includegraphics[width=\textwidth]{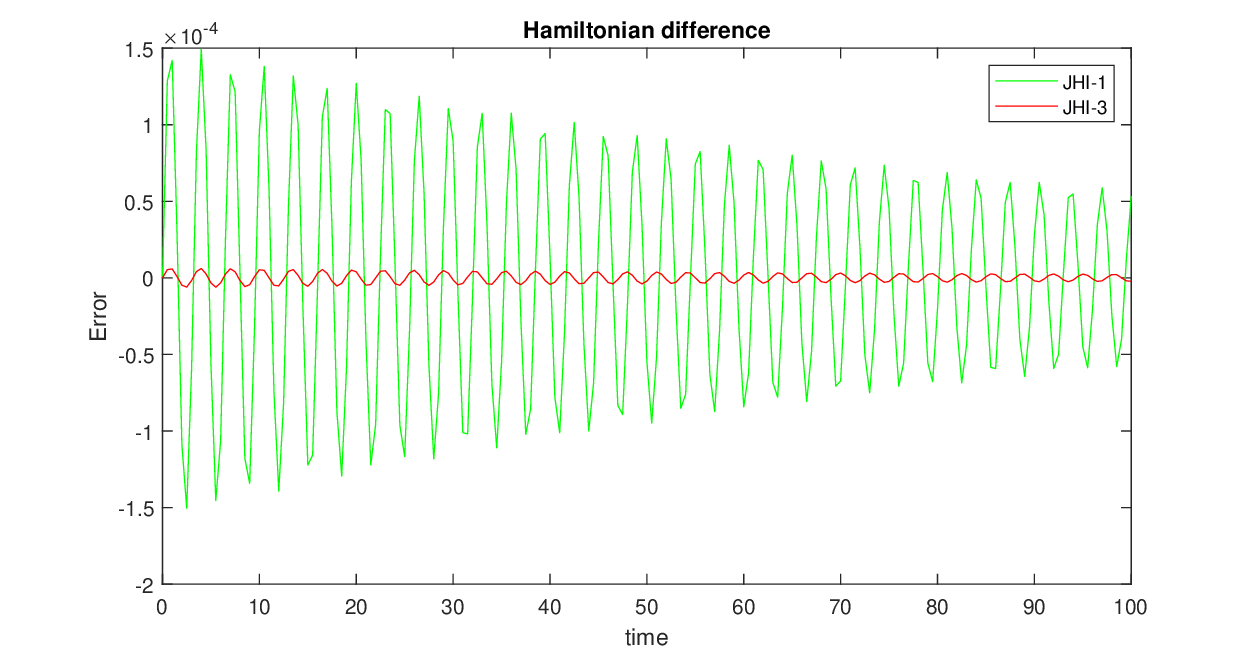}
\caption{Difference in Hamiltonian along the trajectories of first- and third-order JHI methods.}
\label{fig:dampedhamiltonian}
\end{figure}

\subsection{Lotka-Volterra system}\label{sec: exe LV}

The Lotka-Volterra system is a classical example of a Poisson Hamiltonian system modeling the interaction among $n$ species. In its standard form, the system is isolated, meaning $E = 0$. Here, we consider a closed system by promoting the Poisson structure to a Jacobi structure.

First, consider the two-dimensional quadratic bivector field
\[
\Lambda = - a x y \frac{\partial}{\partial x}\wedge \frac{\partial}{\partial y}.
\]

We seek a vector field $E = e_1 \frac{\partial}{\partial x} + e_2 \frac{\partial}{\partial y}$ such that the Jacobi conditions
\[
[\Lambda, \Lambda] = 2 E \wedge \Lambda, \qquad [E, \Lambda] = 0
\]
are satisfied.  

The first condition is trivially satisfied, since $[\Lambda, \Lambda] = 0 = 2 E \wedge \Lambda$ for any $e_1$ and $e_2$; in particular, $\Lambda$ is Poisson. The second condition yields
\begin{align*}
[E, \Lambda] &= \left[e_1 \frac{\partial}{\partial x}, \Lambda \right] + \left[e_2 \frac{\partial}{\partial y}, \Lambda \right] \\
&= a \Big(-y e_1 - x e_2 + xy \frac{\partial e_1}{\partial x} + xy \frac{\partial e_2}{\partial y} \Big) \frac{\partial}{\partial x} \wedge \frac{\partial}{\partial y}.
\end{align*}

A solution is 
\[
e_1 = c xy + dx + fx, \qquad e_2 = c xy + dy - fx, \quad c,d,f \in \mathbb{R}.
\]

Consider the Hamiltonian 
\[
H(x,y) = x - \lambda_1 \log x + y - \lambda_2 \log y,
\]
for which the associated Jacobi Hamiltonian vector field is
\begin{align*}
X_H &= \Lambda(\cdot,dH) - H E(\cdot) \\
&= \left[-axy + a \lambda_2 x - H (cxy + dx + fx) \right] \frac{\partial}{\partial x} + \left[ axy - a \lambda_1 y - H (cxy + dx - fx) \right] \frac{\partial}{\partial y}.
\end{align*}

The corresponding homogeneous Poisson structure is
\[
\Pi = -\frac{a xy}{t} \frac{\partial}{\partial x} \wedge \frac{\partial}{\partial y} + (cxy + dx + fx) \frac{\partial}{\partial t} \wedge \frac{\partial}{\partial x} + (cxy + dy - fy) \frac{\partial}{\partial t} \wedge \frac{\partial}{\partial y}.
\]

To compute a symplectic bi-realization, we assume $x, y > 0$ so that the rank of $\Pi$ is constant. A Casimir function for the homogeneous Poisson structure is
\[
Z(x,y,t) = \log t + \frac{c}{a} (x-y) + \beta \log x - \gamma \log y, \qquad \beta = \frac{d-f}{a}, \ \gamma = \frac{d+f}{a}.
\]

Setting $u = \log x$, $v = \log y$, we have
\[
\{u,v\} = -\frac{axy}{t} \frac{1}{x} \frac{1}{y} = -\frac{a}{t} \neq 1,
\]
so we introduce a canonical transformation $(u,v) \mapsto (U,V)$ satisfying
\[
(U_u V_v - U_v V_u) \{u,v\} = 1.
\]

Setting $U = u = \log x$, we obtain $V_v = -t/a$. Using $Z$, we have
\[
t = \exp\left( Z - \frac{c}{a} (e^u - e^v) - \beta u + \gamma v \right) =: \phi(Z,u,v),
\]
and
\[
V(U,v,Z) = -\frac{1}{a} \int_0^v \phi(Z,u,s) \, ds.
\]

For simplicity, we set $c=0$ and $\beta =0$ (i.e., $d=f$), giving
\[
V = \frac{t}{a\gamma} (y^{-\gamma} - 1).
\]

Thus, the transformation $F(x,y,t) = (U,V,Z)$ satisfies $\{U,V\} = 1$, and a symplectic bi-realization for the canonical Poisson structure is
\[
\alpha_\text{can}(U,V,Z,\eta_U,\eta_V,\eta_Z) = \big(U - \tfrac{1}{2} \eta_V, V + \tfrac{1}{2} \eta_U, Z\big),
\]
so that for the original Poisson structure, $\alpha = F^{-1} \circ \alpha_\text{can} \circ (F^{-1})^*$. The corresponding cotangent transformation is obtained via $J^T \eta = \xi$, with
\[
J = \begin{bmatrix}
\frac{1}{x} & 0 & 0 \\
0 & -\frac{t}{a} y^{-\gamma - 1} & \frac{y^{-\gamma}-1}{a \gamma} \\
0 & -\frac{\gamma}{y} & \frac{1}{t}
\end{bmatrix}, \qquad \eta = J^{-T} \xi.
\]

Explicitly,
\[
\eta_U = x \xi_x, \quad
\eta_V = - a \gamma \xi_t - \frac{a y}{t} \xi_y, \quad
\eta_Z = t y^{-\gamma} \xi_t + \frac{y}{\gamma} (y^{-\gamma}-1) \xi_y,
\]
and
\[
\alpha_\text{can} \circ (F^{-1})^* = \big(\log x + \tfrac12(a\gamma \xi_t + \frac{ay}{t} \xi_y), \ t\frac{y^{-\gamma}-1}{a\gamma} + \tfrac12 x \xi_x, \ \log t - \gamma \log y \big).
\]

Finally, the bi-realization for the original coordinates is
\begin{align*}
x' &= e^{U'} = x \, e^{\frac12 (d+f) \xi_t + \frac{ay \xi_y}{t}}, \\
y' &= (1-C)^{1/\gamma}, \\
t' &= e^{Z'} (1-C), \qquad
C = a \gamma V' e^{-Z'}.
\end{align*}

\paragraph{Numerical examples.}
For the numerical simulations, consider the parameters $\lambda_1 = 3$, $\lambda_2 = 4$ for the Hamiltonian, and $a=1$, $c=0$, $d=f=1$ for the Jacobi structure. The vector field in the domain $(x,y) \in [0,12]\times[0,12]$ is shown in Figure~\ref{fig:L-V_VF}. Unlike the classical Lotka-Volterra system where all trajectories are closed around the equilibrium, here the orbits are only closed near the equilibrium, while trajectories further away tend toward zero.

\begin{figure}[!ht]
\centering
\includegraphics[width=11cm]{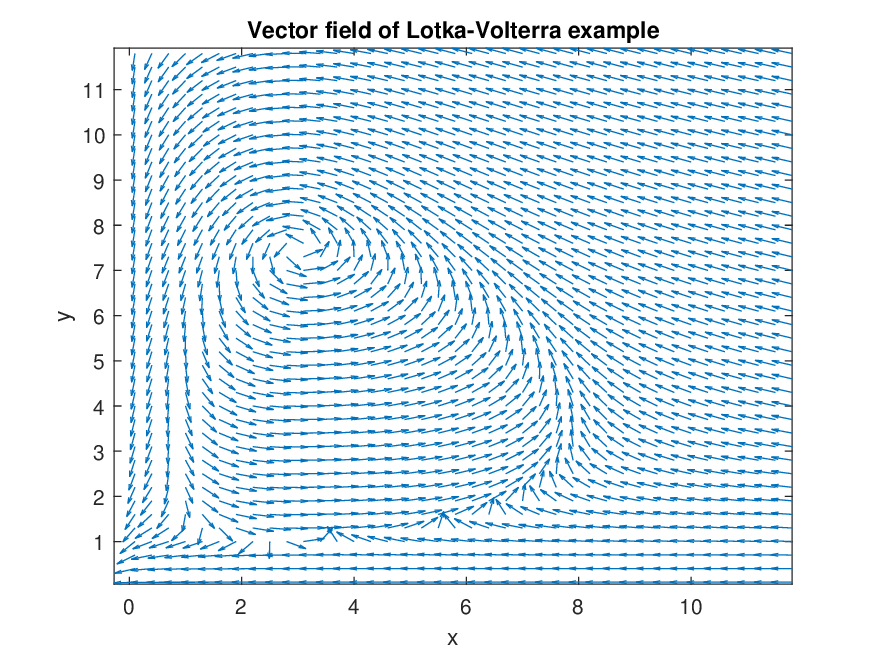}
\caption{Vector field of the modified Lotka-Volterra problem in 2 dimensions.}
\label{fig:L-V_VF}
\end{figure}

To compute trajectories, we use a time span $[0,5]$ with step size $\Delta s = 0.05$ and initial condition $x_0=(4,2)$. Figure~\ref{fig:L-Vtrajectory} compares the trajectories obtained by the first-order Jacobi Hamiltonian Integrator (JHI-1) and the classical RK-2 method. We observe that the JHI-1 method preserves the closed trajectory, while RK-2 does not.

\begin{figure}[!ht]
\centering
\includegraphics[width=11cm]{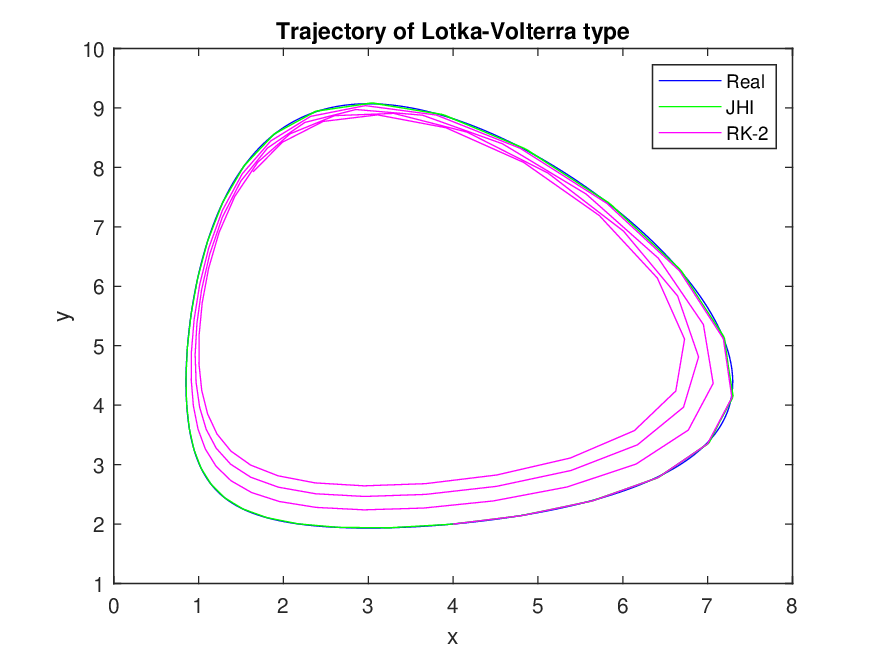}
\caption{Comparison between the true trajectory and those obtained by JHI-1 and RK-2 methods for the modified Lotka-Volterra system.}
\label{fig:L-Vtrajectory}
\end{figure}

The numerical order of JHI-1 is shown in Table~\ref{Tab:LVorder1}. As expected, the method demonstrates second-order convergence.

\begin{table}[!ht]
\centering
\begin{tabular}{c c c}
\hline
$\Delta s$ & $\|\cdot\|_2$ error & Order \\ \hline
$2^{-2}$ & $4.60 \times 10^{-1}$ & -- \\
$2^{-3}$ & $1.48 \times 10^{-1}$ & 1.64 \\
$2^{-4}$ & $3.83 \times 10^{-2}$ & 1.95 \\
$2^{-5}$ & $9.70 \times 10^{-3}$ & 1.99 \\
$2^{-6}$ & $2.40 \times 10^{-3}$ & 2.00 \\
$2^{-7}$ & $6.00 \times 10^{-4}$ & 2.00 \\
$2^{-8}$ & $1.50 \times 10^{-4}$ & 2.00 \\
$2^{-9}$ & $3.70 \times 10^{-5}$ & 2.00 \\ \hline
\end{tabular}
\caption{Error and numerical order of JHI-1 for the modified Lotka--Volterra problem.}
\label{Tab:LVorder1}
\end{table}

The system is not strictly Hamiltonian for the Poisson structure $\Lambda$, but it is instead for the Jacobi structure $(\Lambda, E)$, and so the Hamiltonian is not preserved exactly. In this case,
\[
E(H) = 2 c xy + d(x+y) + f(x-y) - \lambda_1 (c y + d + f) - \lambda_2 (c x + d - f) = 2 (x-3),
\]
so it varies along trajectories. Figure~\ref{fig:HamilLV} shows the Hamiltonian error over a time span $[0,10]$ with step $\Delta s = 0.01$. The error is periodic due to the closed trajectory and remains on the order of $10^{-4}$, demonstrating the good long-time behavior of JHI-1.

\begin{figure}[!ht]
\centering
\includegraphics[width=11cm]{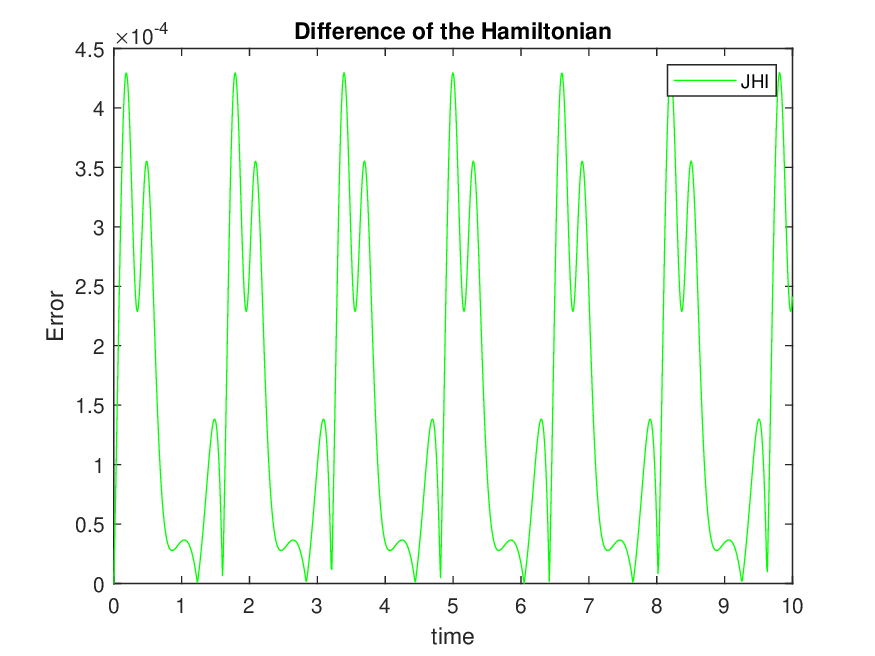}
\caption{Difference in the Hamiltonian along the JHI-1 trajectory for the modified Lotka-Volterra system.}
\label{fig:HamilLV}
\end{figure}

\subsection{Rigid-body rotation}\label{sec: exe RBR}

Finally, we consider a three-dimensional Jacobi system modeling rigid-body rotation with a linear bivector. This example highlights the application of JHI to higher-dimensional systems with polynomial Jacobi vector fields and illustrates the limitations of approximated bi-realizations.

Consider the linear bi-vector field
\[
\Lambda = x_3 \frac{\partial}{\partial x_2}\wedge \frac{\partial}{\partial x_1} + x_2 \frac{\partial}{\partial x_1}\wedge \frac{\partial}{\partial x_3} + x_1 \frac{\partial}{\partial x_3}\wedge \frac{\partial}{\partial x_2},
\]
and let $E = f_1 \frac{\partial}{\partial x_1} + f_2 \frac{\partial}{\partial x_2} + f_3 \frac{\partial}{\partial x_3}$.

To define a Jacobi structure $(\Lambda,E)$, the Jacobi conditions
\[
[\Lambda,\Lambda] = 2 E\wedge \Lambda, \quad [E,\Lambda] = 0
\]
must hold. Since $\Lambda$ is Poisson, $[\Lambda,\Lambda]=0$, so $2E\wedge \Lambda = 0$ implies
\[
f_1 x_1 + f_2 x_2 + f_3 x_3 = 0.
\]

A solution is given by
\begin{align*}
f_1 &= a_1 x_2 x_3 + d_3 x_2 + d_1 x_3,\\
f_2 &= a_2 x_1 x_3 - d_3 x_1 + d_2 x_3,\\
f_3 &= a_3 x_1 x_2 - d_1 x_1 - d_2 x_2,
\end{align*}
with $a_1+a_2+a_3 = 0$, which satisfies both Jacobi conditions.

The Hamiltonian function is
\[
H(x) = \frac{I_2+I_3}{2} x_1^2 + \frac{I_1+I_3}{2} x_2^2 + \frac{I_1+I_2}{2} x_3^2,
\]
with $I_1,I_2,I_3\in\mathbb{R}$, the inertia values, giving the Jacobi Hamiltonian vector field
\begin{align*}
X_H &= \Lambda(\cdot,dH) - H E(\cdot) \\
&= \left[(I_2-I_3)x_2x_3 - H f_1\right]\frac{\partial}{\partial x_1} + \left[(I_3-I_1)x_1x_3 - H f_2\right]\frac{\partial}{\partial x_2} + \left[(I_1-I_2)x_1x_2 - H f_3\right]\frac{\partial}{\partial x_3}.
\end{align*}

The associated homogeneous Poisson structure is
\[
\Pi = \frac{x_3}{t}\frac{\partial}{\partial x_2}\wedge \frac{\partial}{\partial x_1} + \frac{x_2}{t} \frac{\partial}{\partial x_1}\wedge \frac{\partial}{\partial x_3} + \frac{x_1}{t} \frac{\partial}{\partial x_3}\wedge \frac{\partial}{\partial x_2} + \frac{\partial}{\partial t}\wedge(f_1 \frac{\partial}{\partial x_1} + f_2 \frac{\partial}{\partial x_2} + f_3 \frac{\partial}{\partial x_3}).
\]

Since the rank of $\Pi$ is 2, there is one Casimir function
\[
C_1 = x_1^2 + x_2^2 + x_3^2.
\]

For this system, an exact symplectic bi-realization is difficult to obtain analytically due to the quadratic dependence of $f_i$ on the coordinates. To apply the Jacobi Hamiltonian integrator, we construct an \textbf{approximated symplectic bi-realization} by linearizing the cotangent contribution in $\xi$, which preserves the canonical Poisson brackets up to first order in $\xi$. This approach ensures that the resulting numerical integrator retains the geometric structure of the Jacobi system locally and provides accurate trajectories for short- to medium-term simulations, even if global conservation properties are only approximate.

The approximated symplectic bi-realization is defined as
\begin{align*}
\alpha^1(x,t,\xi_x,\xi_t) &= x_1 + \frac12\left(\frac{x_3}{t}\xi_2 + f_1 \xi_t - \frac{x_2}{t}\xi_3\right), &
\beta^1(x,t,\xi_x,\xi_t) &= x_1 - \frac12\left(\frac{x_3}{t}\xi_2 + f_1 \xi_t - \frac{x_2}{t}\xi_3\right),\\
\alpha^2(x,t,\xi_x,\xi_t) &= x_2 + \frac12\left(\frac{x_1}{t}\xi_3 + f_2 \xi_t-\frac{x_3}{t}\xi_1 \right), &
\beta^2(x,t,\xi_x,\xi_t) &= x_2 - \frac12\left(\frac{x_1}{t}\xi_3 + f_2 \xi_t-\frac{x_3}{t}\xi_1 \right),\\
\alpha^3(x,t,\xi_x,\xi_t) &= x_3 + \frac12\left(\frac{x_2}{t}\xi_1 + f_3 \xi_t - \frac{x_1}{t}\xi_2\right), &
\beta^3(x,t,\xi_x,\xi_t) &= x_3 - \frac12\left(\frac{x_2}{t}\xi_1 + f_3 \xi_t - \frac{x_1}{t}\xi_2\right),\\
\alpha^4(x,t,\xi_x,\xi_t) &= t - \frac12(f_1 \xi_1 + f_2 \xi_2 + f_3 \xi_3), &
\beta^4(x,t,\xi_x,\xi_t) &= t + \frac12(f_1 \xi_1 + f_2 \xi_2 + f_3 \xi_3).
\end{align*}

\begin{remark}
The classical rigid-body rotation can be transformed into a Lie-Poisson formulation. However, with quadratic $f_i$, this transformation is generally not possible; it only reduces to Lie-Poisson if $a_1=a_2=a_3=0$. The approximated bi-realization allows the JHI method to be applied while preserving the local geometric structure of the Jacobi system.
\end{remark}

\paragraph{Numerical examples.}

We first consider the parameters $I_1=5$, $I_2=I_3=10$, $a_1=a_2=0.2$, $a_3=-0.4$, $d_1=d_2=d_3=0.1$. The simulation uses a time span $[0,2]$, step $\Delta s=0.005$, and initial condition $x_0=(1,1,1)$. Figure~\ref{fig: RBR trajectory 1} shows that the JHI method coincides with the true trajectory.

\begin{figure}[!ht]
\centering
\includegraphics[width=10cm]{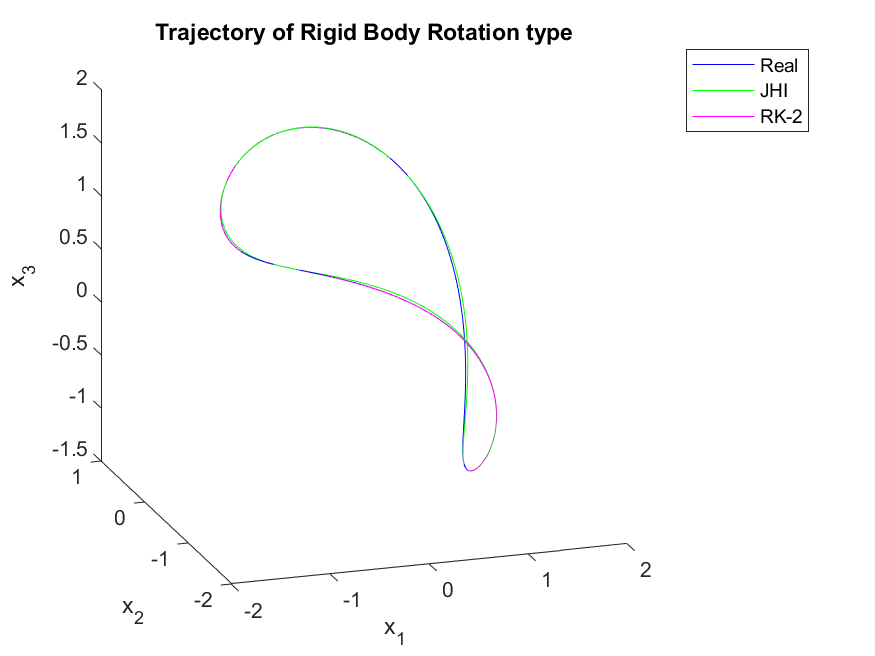}
\caption{Comparison of the true trajectory with JHI and RK-2 for the modified rigid-body rotation with $I_1=5$, $I_2=I_3=10$.}
\label{fig: RBR trajectory 1}
\end{figure}

Next, we test a more unstable configuration: $I_1=1$, $I_2=\pi$, $I_3=10$, with the same Jacobi parameters. Figure~\ref{fig: RBR trajectory 2} shows that JHI trajectories remain closed but do not exactly coincide with the true solution due to the approximated bi-realization, while RK-2 exhibits better local accuracy.

\begin{figure}[!ht]
\centering
\includegraphics[width=10cm]{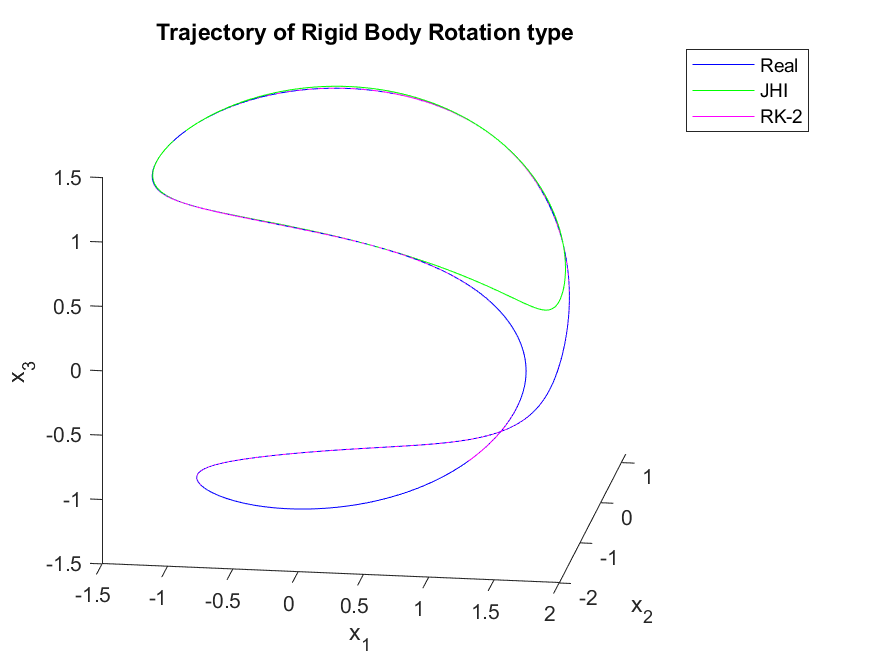}
\caption{Comparison of the true trajectory with JHI and RK-2 for the modified rigid-body rotation with $I_1=1$, $I_2=\pi$, $I_3=10$.}
\label{fig: RBR trajectory 2}
\end{figure}

The Hamiltonian is not conserved exactly, since
\begin{align*}
    E(H) =& x_1x_2x_3\left(a_1(I_2+I_3) + a_2(I_1+I_3) + a_3(I_1+I_2) \right) \\
    &+ d_3x_1x_2(I_2-I_1)+ d_1x_1x_3(I_3-I_1) + d_2x_2x_3(I_3-I_2).
\end{align*}

For the first configuration ($I_1=5$, $I_2=I_3=10$), the Hamiltonian error over the trajectory remains bounded, as shown in Figure~\ref{fig: Hamil RBR}.

\begin{figure}[!ht]
\centering
\includegraphics[width=\textwidth]{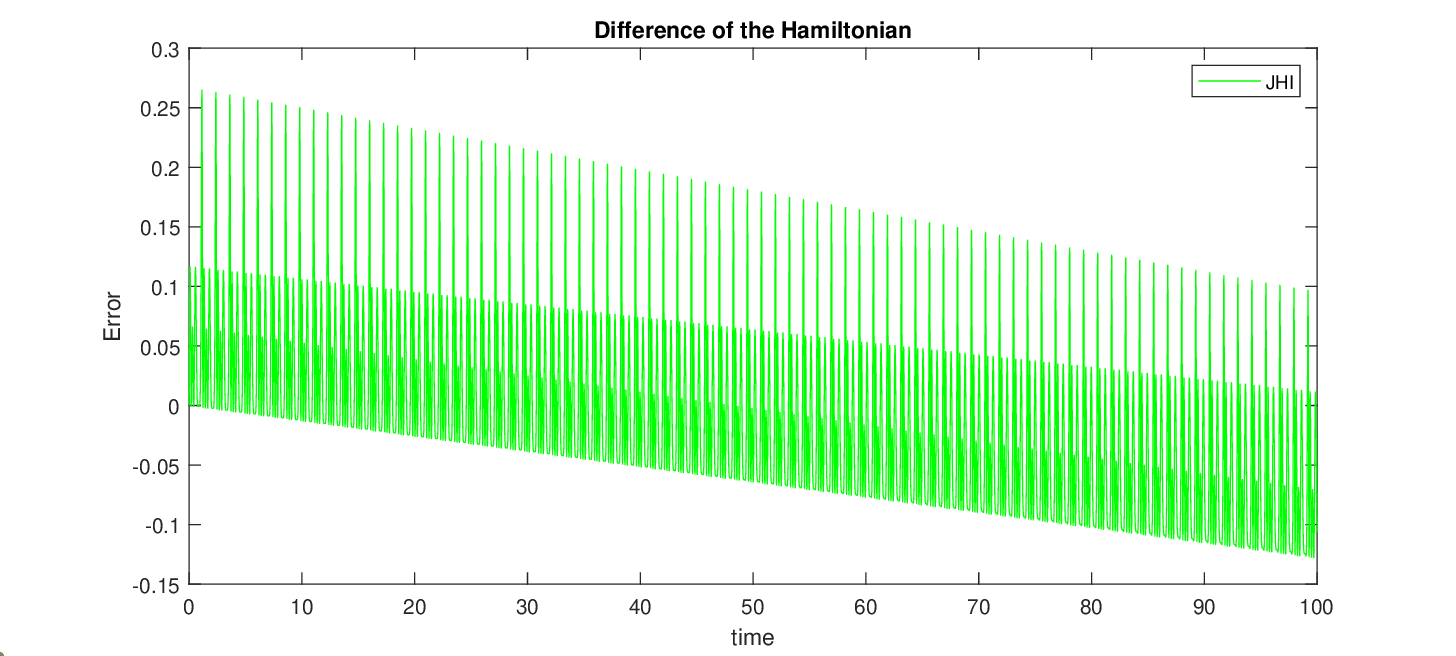}
\caption{Difference in the Hamiltonian along the JHI trajectory for the modified rigid-body rotation with $I_1=5$, $I_2=I_3=10$.}
\label{fig: Hamil RBR}
\end{figure}

\section{Conclusion}

In this work, we have presented a systematic framework for the construction of Jacobi Hamiltonian integrators (JHIs), exploiting the geometric relationship between Jacobi manifolds and homogeneous Poisson manifolds. By applying Poissonization techniques and constructing homogeneous symplectic bi-realizations, we demonstrated that the JHI construction reduces to the creation of homogeneous Poisson Hamiltonian integrators. This approach allows established geometric techniques from the Poisson and symplectic contexts to be adapted to Jacobi systems while preserving the underlying homogeneous structures.

A central contribution of this work is the explicit construction of JHIs of arbitrary order using a recursive algorithm based on the Magnus expansion and exact homogeneous Lagrangian bisections. The resulting methods preserve intrinsic homogeneity, ensuring that the integrators generate well-defined Jacobi maps on the original manifold. Backward error analysis shows that the numerical flow produced by a JHI coincides, up to the expected order, with the exact flow of a modified Jacobi Hamiltonian system, guaranteeing long-time qualitative fidelity.

The numerical experiments illustrate the advantages of JHIs over classical time integration methods. Low-dimensional examples confirm the theoretical order of accuracy, highlight the role of symplectic bi-realizations, and demonstrate the method's robustness when approximate realizations are employed. Applications to classical models further validate the geometric consistency and competitiveness of JHIs, with certain examples reproducing the continuous flow exactly and preserving geometric features such as Hamiltonian or Casimir functions.

Overall, this work establishes Jacobi Hamiltonian integrators as a natural extension of geometric integration techniques to Jacobi manifolds. The proposed framework opens several avenues for future research, including the systematic construction of approximate symplectic bi-realizations, adaptive time-stepping strategies, and applications to larger and higher-dimensional systems. These developments are expected to further strengthen the role of geometric numerical integration in the study of nonconservative and contact-type Hamiltonian dynamics.

\subsubsection*{Acknowledgments} The authors acknowledge financial support by the Centre for Mathematics of the University of Coimbra (CMUC, https://doi.org/10.54499/UID/00324/2025) under the Portuguese Foundation for Science and Technology (FCT), through Grants UID/00324/2025, and UID/PRR/00324/2025. Gon\c calo Inoc\^encio Oliveira acknowledges FCT for support under the Ph.D.
Scholarship 2024.00328.BD.

\newpage
\bibliographystyle{plain}
\bibliography{biblio}
\end{document}